\documentclass[a4paper,12pt]{amsart}
\usepackage{fullpage}
\usepackage{tikz}

\usepackage{enumerate}
\usepackage{amsmath,amscd}
\usepackage{upgreek}

\usepackage{amsthm, stmaryrd}
\usepackage{hyperref}

\usepackage{amssymb,MnSymbol}
\usepackage{latexsym}

\newtheorem{Theorem}{Theorem}

\newtheorem{Definition}[Theorem]{Definition}
\newtheorem{Proposition}[Theorem]{Proposition}
\newtheorem{Lemma}[Theorem]{Lemma}
\newtheorem{Corollary}[Theorem]{Corollary}
\newtheorem{Remark}[Theorem]{Remark}

\newtheorem{Convention}[Theorem]{Convention}

\def\B{{\mathcal B}}

\def\R{{\mathcal R}}

\def\CC{{\mathbb C}}

\def\ZZ{{\mathbb Z}}

\def\RR{{\mathbb R}}

\def\sll{\mathfrak{sl}}
\def\sl2{\sll_2(\CC)}
\def\sln{\sll_N (\CC)}

\def\SLL{\operatorname{SL}}
\def\SL2{\SLL_2(\CC)}
\def\SLn{\SLL_N (\CC)}
\def\SUn{\operatorname{SU}(N)}

\def\tor{\operatorname{tor}}
\def\TOR{\operatorname{TOR}}
\def\Ad{\operatorname{Ad}}

\def\dG{{d}}
\def\rG{r}
\def\bS{b}
\def\b1{k}
\def\Rg{\R^*}

\keywords{Representation varieties, volume forms}
\begin{document}

\title[Volume forms on representation varieties]{Holomorphic volume forms on representation varieties of surfaces with boundary}


\author{Michael Heusener}
\address{Universit\'e Clermont Auvergne, CNRS, Laboratoire de Mathématiques Blaise Pascal, F-63000 Clermont-Ferrand, France}
\email{michael.heusener@uca.fr.}

\author{Joan Porti} 
\thanks{Both authors partially supported by  grant 
FEDER-Meic MTM2015--66165--P and by the Laboratoire de 
 Math\'{e}matiques Blaise Pascal of the Universit\'{e} Clermont Auvergne}
\address{ Departament de Matem\`atiques, Universitat Aut\`onoma de Barcelona, 
08193 Cerdanyola del Vall\`es, Spain, and 
Barcelona Graduate School of Mathematics (BGSMath) }
\email{porti@mat.uab.cat}

\date{\today}

\maketitle

\begin{abstract}
For closed and oriented hyperbolic surfaces, a formula of Witten establishes an equality
between two volume forms on the space of representations of the surface in a semisimple Lie group. 
One of the forms is a 
Reidemeister torsion, the other one  is the  power of 
the Atiyah-Bott-Goldman symplectic form.
We introduce an holomorphic volume form on the space of representations of the circle,
so that, for surfaces with boundary, it  appears as peripheral term in the generalization of Witten's formula. 
We compute explicit volume and symplectic forms for some simple surfaces and for the Lie group $\operatorname{SL}_N(\CC)$.
%
\end{abstract}

%
%

\section{Introduction}
\label{sec:intro}

Along this paper $S=S_{g,\bS}$ denotes a compact, oriented, connected surface with nonempty boundary, of genus $g$
and with 
$\bS\geq 1$ boundary components. 
We assume that $\chi(S)=2-2g-\bS<0$.
The fundamental group $\pi_1(S)$ is a free group $F_\b1$ 
of rank $\b1=1-\chi(S)\geq 2$.

Let $G$ be a connected, semisimple, complex, linear group
with compact real form $G_\RR$, e.g.~$G=\SLn$ and 
$G_\RR=\SUn$. 
We  also assume that $G$ is simply connected; notice that since $\pi_1(S)$ is free, their 
representations lift to the universal covering of the Lie group.

Fix  a nondegenerate symmetric bilinear $G$-invariant  form on the Lie algebra
$$
\B\colon \mathfrak{g}\times \mathfrak{g} \to \mathbb{C}\, ,
$$
such that the restriction of $\B$ to $\mathfrak{g}_\RR$,  the Lie algebra of $G_\RR$, is positive definite. This means that $\B$ is a negative multiple of the Killing form.

Let $\R(S,G)$ denote the set of
conjugacy classes of representations of $\pi_1(S)\cong F_\b1$ into $G$.
We are only interested in irreducible representations for which the centralizer coincides with the center of $G$. 
Following Johnson and Millson \cite{JohnsonMillson} we call such representations \emph{good} (see Definition~\ref{Definition:good}), and we use the notation  
$\Rg(S,G)$
to denote the corresponding  open  subset of $\R(S,G)$.

For a  closed surface $\Sigma$, the bilinear form $\B$  induces two $\mathbb{C}$-valued differential forms
on $ \Rg(\Sigma,G)$, a holomorphic volume form $\Omega_\Sigma$ 
defined as a Reidemeister torsion and the Atiyah-Bott-Goldman (holomorphic) symplectic  form $\omega$.
Witten has shown the following theorem for compact groups, here we state its complexification:

\begin{Theorem}[Witten, \cite{Witten}]
\label{Theorem:Witten}
If $\Sigma$ is a closed, oriented and hyperbolic surface,  then
$$
\Omega_\Sigma=\frac{\omega^{n}}{n!}
$$
on $\Rg(\Sigma,G)$,
where $n=\frac12\dim \Rg(\Sigma,G)=-\frac12\chi(\Sigma)\dim(G)$.
\end{Theorem}

For surfaces with boundary $S$, we need to consider also   $\R(S,\partial S,G)_{\rho_0}$,
the relative set of conjugacy classes of representations
(for each peripheral curve  we require its image to be in a fixed conjugacy class), see Subsection~\ref{subssec:relative}. 
Let 
 $\Rg(S,\partial S,G)_{\rho_0}$ denote the corresponding open subset of good representations.
The holomorphic volume form $\Omega_S$ is  defined on $\Rg(S,G)$ but the holomorphic symplectic form $\omega$ is defined on
$\Rg(S,\partial S,G)_{\rho_0}$.
To relate both spaces and both forms,  we need to deal with each component of $\partial S$, which are circles.

We identify
the 
variety of representations of the circle $S^1$ with $G$,
by mapping each representation to the image of a fixed generator of $\pi_1(S^1)$. We restrict to regular representations,
namely that map the generator of $\pi_1(S^1)$ to regular elements. Then
$
\R^{reg}(S^1,G)\cong G^{reg}/ G
$.
Using that $G$ is simply connected (see Remark~\ref{rem:Gnonsimplyconnected} when $G$ is not simply connected), 
one of the consequences of Steinberg's theorem \cite{SteinbergIHES} is that 
$$
 \R^{reg}(S^1,G)\cong   G^{reg}/ G\cong \CC^r, 
$$
where $r=\operatorname{rank} G$, and that 
there is a natural isomorphism (Corollary~\ref{coro:Steinberg}):
$$
H^1(S^1; \Ad\rho)\cong T_{  [\rho] }  \R^{reg}(S^1,G)  .
$$
In Section~\ref{sec:peripheral} we show the existence of a form
$\nu\colon\bigwedge^{r} H^1( S^1,\Ad\rho)\to \CC$ defined by the formula
\begin{equation}
   \label{eqn:peripheral0}
\nu (\wedge\mathbf v)=\pm \sqrt{ \TOR(S^1,\operatorname{Ad}\rho, \mathfrak{o},\mathbf{u},\mathbf{v} )
 \, \langle \wedge\mathbf v,  \wedge\mathbf u \rangle } .
\end{equation}
Here $\mathbf{u}$ and $\mathbf{v }$ denote bases of $H^0( S^1,\Ad\rho)$ and $H^1( S^1,\Ad\rho) $ respectively
and  $\wedge\mathbf u$ and $\wedge \mathbf v$ 
their exterior product. Moreover, $ \TOR $ denotes the Turaev's sign refined torsion,
$\mathfrak{o}$  an \emph{homology orientation} of 
$H^*(S^1;\RR)$ (see Section~\ref{sec:sgn_ref_tor}),  and $\langle.\,,.\rangle$ the duality pairing  $ H^1( S^1,\Ad\rho)\times  H^0( S^1,\Ad\rho)\to \CC$.
We prove in Lemma~\ref{lemma:independent}  that the value $\nu(\wedge\mathbf v)\in\mathbb C$ is  independent of $\mathbf{u}$.

Steinberg theorem (\cite{SteinbergIHES}, see also  \cite{Steinberg,Popov}) provides an isomorphism
\[
(\sigma_1,\ldots,\sigma_r)\colon G^{reg}/ G\overset{\cong}    \longrightarrow \CC^r,
\]
where $\sigma_1,\ldots, \sigma_r$ denotes a system of fundamental characters of $G$,
which also proves the isomorphism $G{\sslash} G   \cong \CC^r$.

When $G=\SLn$, then $r=N-1$ and
$(\sigma_1,\ldots,\sigma_r)$ are the coefficients of the characteristic polynomial.

\begin{Proposition}
\label{prop:Steinberg}
When $G$ is simply connected, then
 $$
 \nu= \pm C\, d\,\sigma_1\wedge\cdots \wedge d\,\sigma_r,
 $$
 for some constant $C\in\CC^*$ depending on $G$ and $\B$. 
 In addition, for $G=\SLn$ and  
$\mathcal{B}(X,Y)=- \operatorname{tr}(XY)$
for $X,Y\in \sln$,
$$C=\pm (-1)^{(N-1)(N-2)/4}\sqrt{N}.$$
 \end{Proposition}

Let $\rho_0\in \Rg(S,G) $   be \emph{$\partial$-regular}, i.e.\ the image of each peripheral curve is a regular element of $G$
(Definitions~\ref{Definition:regular} and \ref{Definition:partialregular}).  
We have an exact sequence (Corollary~\ref{coro:Steinberg}):
$$
0\to T_{[\rho]}\R^*(S,\partial S,G)_{\rho_0} \to T_{[\rho]}\R^*(S,G) \to \bigoplus_{i=1}^b T_{[\rho(\partial_i)]} \R^ {reg}(\partial_i, G)
\to 0,
$$
where $\partial S=\partial_1\sqcup\cdots\sqcup \partial_\bS$ denote the boundary components of $S$.
For a $\partial$-regular  representation $\rho\colon\pi_1(S)\to G$ we let  $\nu_i$ denote the 
 form corresponding to the restriction $\rho|_{\pi_1(\partial_i)}\colon \pi_1(\partial_i)\to G$
 as in \eqref{eqn:peripheral0}  on $\partial_i\cong S^1$.
 Set $\dG=\dim G$, $r=\operatorname{rank} G$, and $\bS>0$ be the number of components of $\partial S$.
The following generalizes  Theorem~\ref{Theorem:Witten} to surfaces with boundary, \cite{Witten} see also
\cite[Theorem~5.40]{BL}.

 \begin{Theorem} 
 \label{Thm:vol}
 Let $\rho_0\in R(S,G)$ be a good, $\partial$-regular representation.
Then on  $T_{[\rho_0]} \Rg(S,G)$ we have:
  $$
  \Omega_{\pi_1( S)}=\pm
  \frac{\omega^n}{n!}\wedge \nu_1\wedge\cdots\wedge\nu_\bS,
  $$
  where $n = \frac{1}{2} \dim \R^*(S,\partial S,G)_{\rho_0} = \frac12(-\chi(S)\, \dG - b\, r)$.
\end{Theorem}

Notice that we write $\Omega_{\pi_1( S)}$ instead of $\Omega_{ S}$, as the simple homotopy type of $S$ only depends on $\pi_1(S)$.
Following Witten \cite{Witten} in the closed case,  
the proof of Theorem~\ref{Thm:vol} is based on Franz-Milnor duality for Reidemeister torsion.

The formula of Theorem~\ref{Thm:vol} is homogeneous in the bilinear form $\B\colon\mathfrak{g}\times\mathfrak{g}\to\CC $: if $\B$ is replaced by $\lambda^2\B$ for some $\lambda\in\CC^* $, then
$\omega$ is replaced by $\lambda^2\omega $, $\nu_i$ by $\lambda^\rG\, \nu_i$ and 
$ \Omega_{\pi_1( S)}$ by $ \lambda^{ 2\, n+ b\, r }\, \Omega_{\pi_1(S)}$, as $ { 2\, n+ b\, r }=-\chi(S)d=\dim \R(S,G) $. 
%
%
%

 \medskip
 
We focus now on $G=\SLn$, which is simply connected and has rank $r=N-1$. We fix a bilinear form on the Lie algebra: 
 
%

\begin{Convention}
Along this paper, when $G=\SLn$ we always assume $\mathcal{B}(X,Y)=- \operatorname{tr}(XY)$
for $X,Y\in \sln$.
\end{Convention}

We compute explicit volume forms for spaces of representations of free groups in $\SL2$ and $\SLL_3(\CC)$. We start
with a pair of pants $S_{0,3}$.
The fundamental group $\pi_1(S_{0,3})\cong F_2$ is free on two generators $\gamma_1$ and $\gamma_ 2$. By Fricke-Klein theorem, 
$X(F_2,\SL2)\cong\mathbb C^3$ and the  coordinates are precisely the traces of the peripheral elements $\gamma_1$, $\gamma_2$,  and $\gamma_1\gamma_2$,
denoted by 
$t_1$, $t_2$, and $t_{12}$ respectively. 
In this case the relative character variety 
is just a point, and the symplectic form is trivial. Thus, by applying Theorem~\ref{Thm:vol} and  equality
$
\nu=\pm\sqrt 2 \,  d \operatorname{tr}_\gamma
$ (Proposition~\ref{prop:Steinberg}), we have
$$
\Omega_{F_ 2} =\Omega_{\pi_1(S_{0,3})}=\pm 2\sqrt 2 \,  d\, t_1\wedge d\, t_2\wedge d\, t_{12},
$$
on $\R^*(F_2,\SL2)$.

By~\cite{AcunaMontesinos}, for $k\geq 3$, the $3k-3$ trace functions
$t_1,t_2,t_{12}, t_3,t_{13}, t_{23}, \ldots, t_{k},t_{1k}, t_{2k}$
 define a local parameterization
 \[
T\colon \R^*(F_k,\SL2) \setminus  \operatorname{crit}(T)
\to \CC^{3k-3},
\] 
where $\operatorname{crit}(T) =\bigcup_{i\geq 3}\{ t_{12i}=t_{21i} \}\cup\{t_{12\bar1\bar2}=2\}$.
Here, $t_{i_1\cdots i_l}\colon \R^*(F_k,\SL2)\to\CC$ stands for the trace function 
$ \operatorname{tr}_\gamma$ if $\gamma = \gamma_{i_1}\cdots\gamma_{i_l}$
with the convention $\gamma_{\,\bar i}=\gamma_i^{-1}$.

\begin{Theorem}
\label{thm:F3SL2}
The holomorphic volume form on 
$
\R^*(F_k,\SL2) \setminus \operatorname{crit}(T)
$
is $ \Omega_{F_k}^{\SL2}= \pm T^*\Omega$, where
$$
 \Omega= \pm 2\sqrt{2}\,  d\, t_1\wedge d\, t_2\wedge d\, t_{12} \, \bigwedge_{i=3}^k \sqrt{2} \frac{d\, t_i \wedge d\, t_{1i}\wedge d\, t_{2i} }{t_{12i}-t_{21i}}  .
$$
\end{Theorem}

Next we deal with $\SLL_3(\CC)$.
To avoid confusion with $\SL2$, the trace functions in $\SLL_3(\CC)$ are denoted
by $\uptau_{i_1\cdots i_k}$;
notice that $\uptau_{\bar i}\neq \uptau_i$.
Lawton obtains in \cite{LawtonJA} an explicit description of the variety of characters $X(F_2,\SLL_3(\CC))$. It follows from his result that
$$
\mathcal{T} :=
(\uptau_1, \uptau_{\bar 1},\uptau_2, \uptau_{\bar 2},\uptau_{12}, \uptau_{\bar 1\bar 2}, \uptau_{1\bar 2}, \uptau_{\bar 1 2})\colon
\R^*(F_2,\SLL_3(\CC)) \setminus \{ \uptau_{12\bar 1\bar 2}= \uptau_{21\bar 2\bar 1} \}\to \CC^8
$$
defines a local parameterization. Using the computation of the symplectic form in \cite{LawtonP},
we   prove in Proposition~\ref{Prop:volF2SL3} that, 
on $\R^*(F_2,\SLL_3(\CC) ) \setminus \{   \uptau_{21\bar2\bar1}=  \uptau_{12\bar1\bar2}\}$ the volume form is $\Omega_{F_2}^{\SLL_3(\CC)}= \mathcal{T}^*\Omega$, where
 $$
\Omega=\pm \frac{3\sqrt{-3}}{  \uptau_{21\bar2\bar1}-\uptau_{12\bar1\bar2}}
d\, \uptau_1\wedge d\, \uptau_{\bar 1} \wedge d\, \uptau_2\wedge d\, \uptau_{\bar 2}
\wedge d\, \uptau_{12}\wedge d\, \uptau_{\bar 1\bar 2}\wedge d\, \uptau_{1\bar 2}\wedge d\, \uptau_{\bar 1 2}.
$$ 

This is generalized to a free group of arbitrary rank. First we start with the generic local parameters:
\begin{Proposition}
\label{prop:FkSL3}
For $k\geq 3 $, the $8k-8$ trace functions
\begin{multline*}
  \mathcal T= (\uptau_1, \uptau_{\bar 1},\uptau_2, \uptau_{\bar 2},\ldots,  \uptau_k, \uptau_{\bar k}, \uptau_{12}, \uptau_{\bar 1\bar 2}, \uptau_{13}, \uptau_{\bar 1\bar 3},\ldots , \uptau_{1k}, \uptau_{\bar 1\bar k},
     \uptau_{23}, \uptau_{\bar 2\bar 3},\ldots , \uptau_{2k}, \uptau_{\bar 2\bar k},
  \\
  \uptau_{1\bar 2}, \uptau_{\bar 1 2},  \uptau_{1\bar 3}, \uptau_{\bar 1 3}, \ldots, \uptau_{1\bar k}, \uptau_{\bar 1 k} ) 
\end{multline*}
define a local parameterization 
$
\mathcal{T}\colon  \R^*(F_k,\SLL_3(\CC)) \setminus \operatorname{crit}(\mathcal T)
 \to \CC^{8k-8}, 
$
with
\begin{align*}
  \operatorname{crit}(\mathcal T)&=\bigcup\limits_{i\geq 2}\{ \uptau_{1i\bar 1\bar i}= \uptau_{i1\bar i\bar 1}  \} \cup  \bigcup\limits_{i\geq 3}\{  \varDelta^1_{2i} =0\}, \\
  \varDelta^1_{2i}  &=  (\uptau_{12i}-\uptau_{1i2}) (\uptau_{\bar 1\bar 2\bar i}-\uptau_{\bar 1\bar i\bar 2}) - 
(\uptau_{1\bar 2\bar i}-\uptau_{1\bar i\bar 2}) (\uptau_{\bar 1 2  i}-\uptau_{\bar 1 i 2 }).
\end{align*}

%
%
\end{Proposition}

Next we provide the holomorphic volume form:

  \begin{Theorem}
\label{thm:FkSL3}
The volume form  on 
$\R^*(F_k,\SLL_3(\CC))\setminus   \operatorname{crit}(\mathcal T)$ 
is 
$  \Omega_{F_k}^{\SLL_3(\CC)}   =\pm \mathcal{T}^* \Omega$, for
$$
 \Omega =\omega_{12}\wedge\nu_1\wedge\nu_2\wedge\nu_{12}\bigwedge_{i=3}^k\frac{\omega_{1i}\wedge \nu_i\wedge \nu_{1i}\wedge \nu_{2i} }{3 \varDelta^1_{2i}}
$$
where
\[
\nu_i  = \sqrt{-3}\, d\, \uptau_i\wedge d\, \uptau_{\bar i}, \quad
\nu_{\ell i}  = \sqrt{-3}\, d\, \uptau_{\ell i}\wedge d\, \uptau_{\bar\ell \, \bar i}, \quad
\omega_{1i}  =   \frac1{  {  \uptau_{1i\bar 1\bar i}- \uptau_{i1\bar i\bar 1}  }        } {d\, \uptau_{1\bar i}\wedge d\, \uptau_{\bar 1 i} },
\]
and  $ \varDelta^1_{2i}$ is  as in Proposition~\ref{prop:FkSL3}.
\end{Theorem}


 The paper is organized as follows. In Section~\ref{Section:VR} we review the results on spaces of representations that we need, in particular we 
 describe the relative variety of representations. In Section~\ref{Section:RT} we recall the tools of Reidemeister torsion, including the duality formula,
on which Theorem~\ref{Thm:vol} is based. In Section~\ref{Section:forms} we describe all forms and we prove Theorem~\ref{Thm:vol}.
Section~\ref{Section:SLn} is devoted to formulas for $\SLn$, the form $\nu$ and as well as the volume form for the free groups of rank $2$
in $\SL2$ and $\SLL_3(\CC)$.
In Section~\ref{Section:symplectic} we use Goldman's  formula for the Poisson bracket 
to give the symplectic form  in terms of trace functions for the 
relative varieties of representations of $S_{0,4}$ and $S_{1,1}$ in $\SL2$. Finally, in Section~\ref{section:higherrank}
we compute volume forms on spaces of representations of free groups of higher rank in $\SL2$ and $\SLL_3(\CC)$.

%

\smallskip
\noindent
\textbf{Acknowledgements.} 
We are indebted to
Simon Riche for helpful discussions and  for pointing out Steinberg results to us.

 \section{Varieties of representations}
\label{Section:VR} 

Throughout this article $G$ denotes a simply-connected  semisimple complex linear Lie group.
We let $\dG$ denote the dimension of $G$, and $\rG$ its rank.
Also recall that along this paper $S=S_{g,\bS}$ denotes a compact, oriented, connected surface with nonempty boundary, of genus $g$
and with 
$\bS\geq 1$ boundary components, $\partial S = \partial_1 \sqcup\cdots\sqcup\partial_b$. 
We assume that $\chi(S)=2-2g-\bS<0$.
The fundamental group  of $S$ is a free group $F_\b1$ 
of rank $\b1=1-\chi(S)\geq 2$.

 \subsection{The variety of good representations}
 
The set of representations of $\pi_1(S)\cong F_\b1$ 
in $G$ is 
$$
R(S,G)=\hom(\pi_1(S),G)\cong G^k\, .
$$
It follows from \cite[Chap.4, \S1.2]{OV} that $G$ is algebraic, and hence $R(S,G)$ is an affine algebraic set  (it has a natural 
algebraic structure independent of the choice of the isomorphism
$\pi_1(S)\cong F_k$).

The group $G$ acts on $R(S,G)$ by conjugation and we are interested in the quotient
$$
\R(S,G)=R(S,G)/G\, .
$$
This is not a Hausdorff space, so we need to restrict to representations with some regularity properties.
Following \cite{JohnsonMillson}, we define:

\begin{Definition} A representation $\rho\in R(S,G)$ is \emph{irreducible} if its image is not contained
		in a proper parabolic subgroup of $G$.
\end{Definition}

For $\rho\in R(S,G)$, its centralizer is
$$
Z(\rho)=\{ g\in G \mid g\rho(\gamma)=\rho(\gamma)g,\ \forall\gamma\in\pi_1(S)\}\, .
$$

\begin{Proposition}[Proposition~1.1 of \cite{JohnsonMillson}]
The representation $\rho\in R(S,G)$ is \emph{irreducible} if and only if the orbit 
$\mathcal{O}(\rho)$ is closed in  $R(S,G)$
           and  $Z(\rho)$ is finite.
\end{Proposition}

\begin{Definition}
\label{Definition:good}
A representation $\rho\in R(S,G)$ is
\emph{good} if it is irreducible and its centralizer  
$Z(\rho)$ is the center of the group $G$,
i.e.~$Z(\rho)=Z(G)$. 
\end{Definition}

The set of good representations is denoted by $R^*(S,G)$, and its orbit space by
$\Rg(S,G) = R^*(S,G)/G$. 


\begin{Proposition}[Proposition~1.2 and 1.3 of \cite{JohnsonMillson}] The set of good representations 
  $R^*(S,G)$ is a Zariski open subset of $R(S,G)$.
  Furthermore the action of $G$ on  $R^*(S,G)$ is proper.
\end{Proposition}

The \emph{variety of characters} is the quotient in the algebraic category:
$$
X(S,G)=R(S,G)\sslash G\, .
$$
Namely, it is an algebraic affine set defined by its ring of polynomial functions, as the ring of functions
on $R(S,G)$ invariant by conjugation.

The projection $R(S,G)\to X(S,G)$ factors through 
a surjective map $\R(S,G)\to X(S,G)$. For good representations we have:

\begin{Proposition}
\label{Prop:smoothvariety}
The natural map restricts to an injection 
$$\Rg(S,G)\hookrightarrow X(S,G)$$ 
whose image is a Zariski open subset and a smooth complex manifold.
\end{Proposition}

For the proof, see for instance \cite[\S1]{JohnsonMillson}, or \cite[Proposition~3.8]{Newstead} for injectivity,
as irreducibility is equivalent to stability in GIT \cite[\S1]{JohnsonMillson}.
For smoothnesses see \cite{GoldmanAdvances}.

Given a representation $\rho\in R(S,G)$, the Lie algebra $\mathfrak{g}$ turns into an $\pi_1(S)$-module
via $\operatorname{Ad}\circ \rho$.
If there is no ambiguity this module is denoted just by  $\mathfrak g$,
and the coefficients in cohomology are denoted by
$\operatorname{Ad}\rho$.
\begin{Proposition}
\label{Prop:tangentspace} Let $\rho\in R^*(S,G)$ be a good representation.
Then there is a natural isomorphism
$$T_{[\rho]}\Rg(S,G)\cong H^1(S;\operatorname{Ad}\rho) . $$
In particular the dimension of $\Rg(S,G)$ is  $-\chi(S)\, \dG$.
\end{Proposition}

This proposition can be found for instance in 
\cite[Corollary~50]{Sikora12}, but we sketch the proof  as 
it may be useful for the relative case.

\begin{proof}
Let $Z^1=Z^1(S;\Ad\rho)$
denote the space of crossed morphisms from $\pi_1(S)$
to $\mathfrak{g}$, i.e.~maps 
$d\colon \pi_1(S)\to \mathfrak{g}$
satisfying $d(\gamma\mu)=d(\gamma)+
\operatorname{Ad}_{\rho(\gamma)}d(\mu)$, $\forall\gamma,\mu\in\pi_1(S)$.
Let $B^1=B^1(S;\Ad\rho)$ denote
the subspace of inner crossed morphisms: 
for $a\in\mathfrak{g}$ the corresponding inner morphism
maps $\gamma\in\pi_1(S)$ to 
$\operatorname{Ad}_{\rho(\gamma)}(a)-a$.
Weil's construction identifies 
$Z^1$
with $T_{\rho}R(S,G)$ (usually $Z^1$ is the
Zariski tangent space to a scheme,
possibly non-reduced, but as 
$\pi_1(S)$ is free, $R(S,G)$ is a smooth algebraic variety).
The subspace 
$B^1$
corresponds to the tangent space to the orbit
$\operatorname{Ad}_G(\rho)$. Then, in order to identify
the tangent space to $\Rg(S,G)$ with the cohomology group
$
H^1(S;\Ad\rho)=Z^1/B^1
$,
 we use a slice, for instance
an \'etale slice provided by Luna's theorem \cite[Theorem~6.1]{AlgebraicGeometryIV}, 
or an analytic slice (cf.~\cite{JohnsonMillson}). 
In the setting of a good representation $\rho$,
a slice is a subvariety $\mathcal{S}\subset R(S,G)$
containing $\rho$, invariant by $Z(\rho)=Z(G)$,
such that the conjugation map
$$
G/Z(G)\times \mathcal{S}\to R(S,G)
$$
is locally bi-analytic at $(e,\rho)$ and
the projection $\mathcal{S}\to X(S,G)$ is also
bi-analytic at $\rho$.
(If $\rho$ was not good, we should take care of
the action of $Z(\rho)/Z(G)$. In addition, for $\Gamma$ not a free group
the description is more involved). Then the assertion follows easily
from the properties of the slice.
\end{proof}

\subsection{The relative variety of representations}
\label{subssec:relative}
Let 
$$\partial S=\partial_1\sqcup\cdots\sqcup \partial_\bS$$ 
denote the decomposition in connected components. 
By abuse of notation, we also let $\partial_i$ denote  an element of the fundamental group  represented by the
corresponding  oriented  peripheral curve. This is well defined only up to conjugacy in $\pi_1(S)$,
but our constructions do not depend on the representative in the conjugacy class.

\begin{Definition}[\cite{Kapovich} \S 4.3]
For $\rho_0\in R(S,G)$, the 
 \emph{relative variety of representations} is
\[
\R(S,\partial S,G)_{\rho_0}=\{[\rho]\in \R(S,G)\mid \rho(\partial_i)\in\mathcal{O}(\rho_0(\partial_i)),\ i=1,\ldots,b\}\, .
\]
Here $\mathcal{O} (\rho_0(\partial_i))$ denotes the conjugacy class of $\rho_0(\partial_i)$.
We also denote
$$
\Rg(S,\partial S,G)_{\rho_0}=\R(S,\partial S,G)_{\rho_0}\cap \Rg(S,G)\, .
$$
\end{Definition}

Besides considering good representations, we  restrict our attention to representation which map  peripheral elements to regular elements of $G$.

\begin{Definition}[\cite{Steinberg} \S 3.5]
\label{Definition:regular}
An element $g\in G$ is called \emph{regular} if its centralizer $Z(g)$ has minimal dimension among centralizers
of elements of $G$. Equivalently, its conjugacy class $\mathcal{O}(g)$ has maximal dimension. 
\end{Definition}

This minimal dimension is $\rG$ the rank of $G$ \cite[\S3.5, Proposition~1]{Steinberg}. In $\SLn$, a diagonal matrix is
regular if and only if all eigenvalues are different. 
More generally, $g\in\SLn$ is regular if and only if its minimal polynomial is of degree $N$ \cite[\S3.5, Proposition~2]{Steinberg}.
In particular the companion matrix of a monic polynomial is regular.

\begin{Definition} 
\label{Definition:partialregular}
A representation $\rho\in R(S,G)$ is called \emph{$\partial$-regular} if the elements
$\rho(\partial_1),\dotsc,\rho(\partial_\bS)$ are regular.
\end{Definition}


\begin{Proposition}
\label{prop:relsmooth}
 Let $\rho_0\in R^*( S,G)$ be a good, $\partial$-regular representation.
 \begin{enumerate}[(a)]
  \item $\Rg(S,\partial S, G)_{\rho_0}$ is a complex smooth manifold of dimension 
  \[
  d_0=-\dG\,\chi(S)-\bS\, \rG = \dG (2 g(S) -2)+\bS (\dG-\rG) .
  \]

  \item For $[\rho]\in  \Rg(S,\partial S, G)_{\rho_0} $, there is a natural isomorphism:
   $$T_{[\rho]} \Rg(S,\partial S, G)_{\rho_0}\cong\ker \left(H^1(S;\Ad\rho)\to H^1(\partial S;\Ad\rho )\right)
   .$$
 \end{enumerate}
\end{Proposition}

\begin{proof}
We first show that the map   $H^1(S;\Ad\rho)\to H^1(\partial S;\Ad\rho )$ is a surjection. 
By Poincar\' e duality 
$H^2( S,\partial S;\Ad\rho ) \cong H^0(S;\Ad\rho )\cong \mathfrak{g}^{\operatorname{Ad}\rho(\pi_1(S))}$, that vanishes because
$Z(\rho)$ is finite. Thus, by the long exact sequence of the pair $( S,\partial S)$, the map $H^1(S;\Ad\rho)\to H^1(\partial S;\Ad\rho )$ is a surjection. 


We use a slice at $\rho_0$, $\mathcal{S}\subset R(M)$ as in the proof of Proposition~\ref{Prop:tangentspace}.
The fact that $H^1(S;\Ad\rho)\to H^1(\partial S;\Ad\rho )$ is a surjection
means that the restriction map 
$$
\operatorname{res}\vert _\mathcal{S}\colon\mathcal{S}\to R(\partial S, G)=
\prod_{i=1}^\bS R(\partial_i, G)= G^\bS
$$
 is transverse to the products of orbits by conjugation
$$
O=\prod_{i=1}^\bS\mathcal{O}(\rho(\partial_i))\, .
$$
Namely, $(\operatorname{res}\vert_\mathcal{S})_* (T_{\rho}\mathcal{S})+ T_{\operatorname{res}(\rho)} O= T _{\operatorname{res}(\rho)} G^\bS$. 
It follows from the rank theorem \cite[C.4.1]{Loj} that 
$\mathcal{O}(\rho(\partial_i))\subset G$ is a complex analytic subvariety of dimension 
$\dG-\rG$ because $\rho$ is  $\partial$-regular.
Thus $(\operatorname{res}\vert_\mathcal{S})^{-1}(O)$
is an analytic $\CC$-submanifold of codimension 
$$\dim G^\bS -\dim O
=  \sum_{i=1}^\bS \big(\dim G -\dim \mathcal{O}(\rho(\partial_i))\big)
=
\bS\, \rG\, .
$$
Now the proposition follows from the properties of the slice.
\end{proof}

\subsection{Steinberg map}
\label{subsection:Steinberg}

In order to understand the space of conjugacy classes of regular representations of $\ZZ$
we identify each representation with the image of its generator, so that 
$$
R^{reg}(\ZZ,G)= G^{reg}
\quad\textrm{ and }\quad
\R^{reg}(\ZZ,G)= G^{reg}/G.
$$
Consider the Steinberg map
\begin{equation}
\label{eqn:Steinberg}
(\sigma_1,\cdots,\sigma_r) \colon G  \to \CC^r
\end{equation}
where $\sigma_1,\cdots,\sigma_r$ denote the characters corresponding to a system of fundamental representations (for $\SLn$
those are the coefficients of the characteristic polynomial).

\begin{Theorem}(Steinberg, \cite{SteinbergIHES})
\label{thm:Steinberg}
If $G$ is simply connected, then
 the map \eqref{eqn:Steinberg}
 is a surjection and has a section $ s\colon \CC^r\to G^ {reg}$ so that 
 $s(\CC^r)$ is a subvariety
that 
 intersects each orbit by conjugation in $G^ {reg}$  precisely once.
 \end{Theorem}

For instance, when $G=\SLn$ the section in Theorem~\ref{thm:Steinberg} 
can be chosen to be the companion matrix (see \cite[p.~120]{Steinberg} and \cite[Sec.~4.15]{Humphreys}).

\begin{Corollary}
\label{coro:Steinberg}
If $G$ is simply connected, then:
\begin{enumerate}[(i)]
 \item  The map \eqref{eqn:Steinberg} induces natural isomorphisms between
 the space of regular orbits by conjugation, the variety of characters, and $\CC^r$:
  $$
 \R^ {reg}(S^ 1, G)\cong X(S^ 1, G)\cong   \CC^ r  .
 $$
 \item \label{ii} The Steinberg map induces  a natural isomorphism
 $$
 H^ 1(S^ 1,\Ad\rho)\cong  T_{[\rho]}\R^ {reg}(S^ 1, G)\cong \CC^ r\,.
 $$ 
 Moreover, for each  good, $\partial$-regular representation $\rho_0\in R^*(S,G)$ and
$ [\rho]\in\R^*(S,\partial S,G)_{\rho_0}$ there is an exact sequence 
 $$
0\to T_{[\rho]}\R^*(S,\partial S,G)_{\rho_0} \to T_{[\rho]}\R^*(S,G) \to \bigoplus_{i=1}^b T_{[\rho(\partial_i)]} \R^ {reg}(\partial_i, G)
\to 0
.
 $$
\end{enumerate}
\end{Corollary}

\begin{proof}
 For (i), notice that what we aim to prove is the isomorphism $
 G^ {reg}/G\cong G \sslash  G\cong   \CC^ r$; which is straightforward from the existence of the section 
 in  Theorem~\ref{thm:Steinberg}.
 
For (ii), by the existence of the section we also know that the differential of Steinberg's map
$Z^1(\ZZ,\Ad\rho)\cong \mathfrak g\to \CC^r$ is surjective whenever $\rho$ is regular \cite[\S 4.19]{Humphreys}. In addition
it maps $B^1(\ZZ,\Ad\rho)$ to $0$, because Steinberg map is constant on orbits by conjugation. 
Thus we have a well defined surjection $H^1(S^1,\Ad\rho)\to \CC^r$, which is an isomorphism because of the dimension. 
The exact sequence follows from the long exact sequence in cohomology of the pair $(S,\partial S)$ and the identification
of cohomology groups with tangent spaces, cf.~Proposition~\ref{prop:relsmooth}.
\end{proof}

\begin{Remark}
\label{rem:Gnonsimplyconnected}
When $G$ is not simply connected, then the universal covering $\tilde G\to G$ is finite and $\pi_1(G)$ can be identified with a (finite) central subgroup $Z$ of $\tilde G$. The center of $\tilde G$ acts on the quotient $\tilde G \sslash \tilde G$ and we obtain a commutative diagram
\[
\begin{CD}
 \tilde G @>>> \tilde G \sslash \tilde G \\
 @VVV     @VV{\varphi}V\\
 G @>>>  G \sslash  G 
\end{CD}
\]
where $(G\sslash G,\varphi)$ is a quotient for the action of $Z$ on $\tilde G \sslash \tilde G$ (see \cite[Lemma~2.5]{Popov}).
Notice that $\varphi$ is a finite branched covering.

Then part~(\ref{ii}) of Corollary~\ref{coro:Steinberg} can be easily adapted for those $[g]\in G\sslash G$ 
which are outside the branch set of $\varphi$.
\end{Remark}


 \section{Reidemeister torsion}
 \label{Section:RT} 
 
Let $\rho\in R(S,G)$ be a representation; recall that we consider the action of $\pi_1(S)$ on  $\mathfrak{g}$ via the adjoint of $\rho$.
Most of the results in this section apply not only to  $\mathfrak{g}$ but to its real form
$\mathfrak{g}_{\RR}$, provided that the image of the representation is contained in $G_\RR$.
Recall also that we assume that $\B$ restricted to the compact real form
$\mathfrak{g}_{\RR}$ is positive definite.
 
Consider a cell decomposition $K$ of $S$. 
If $C_*(\widetilde K;\mathbb Z)$ denotes the simplicial chain complex
on the universal covering, one defines
 \begin{equation}
  \label{eqn:chaincplx}
   C^*(K;\Ad\rho)=\hom_{\pi_1(S)}( C_*(\widetilde K;\mathbb Z), \mathfrak{g}) .
 \end{equation}

 We consider the so called geometric basis. Start with a $\B$-orthonormal $\CC$-basis $\{m_1,\ldots,m_{\dG}\}$
of $\mathfrak{g}$. 
For each $i$-cell $e^i_j$ of $K$ we choose a lift $\tilde e^i_j$ to the universal covering $\widetilde{K}$, then 
$$
\mathbf c^i=\{ (\tilde e^i_j )^*\otimes m_k \}_{jk} 
$$
is a basis of $ C^i(K;\Ad\rho)$, called the \emph{geometric basis}. Here, 
$(\tilde e^i_j )^*\otimes m_k\colon C_*(\widetilde K;\mathbb Z)\to \mathfrak{g}$ is the unique
$\pi_1(S)$-homomorphism given by $(\tilde e^i_j )^*\otimes m_k (\tilde e^i_l) =\delta_{jl} m_k$.

On the other hand, if 
$B^i=\operatorname{Im}(\delta\colon C^{i-1}(K;\Ad\rho)\to  C^i(K;\Ad\rho)) $ is the  space of coboundaries, chose 
$\mathbf b^i$ a basis for $B^i\subset C^i$ and chose lift $\widetilde{\mathbf b^i}$  to $C^{i-1}$ by the coboundary map.
For a basis $\mathbf h^i$ of  $H^i(K;\Ad\rho)$, consider also representatives   
$\widetilde{\mathbf h^i}\in C^i(K;\Ad\rho)$. Then the disjoint union
$$ 
\widetilde{\mathbf b^{i+1}}\sqcup  \widetilde{\mathbf h^i} \sqcup {\mathbf b^i}
$$ is also a basis for $C^i(K;\Ad\rho)$.
Notice that we are interested in the case where the zero and second cohomology groups vanish, so we assume that $\widetilde{\mathbf h^0}=\widetilde{\mathbf h^2}=\emptyset$.

\smallskip

Reidemeister torsion is defined as
\begin{equation}
\tor(S,\operatorname{Ad}\rho, \mathbf h^1)= 
\frac{[  \widetilde{\mathbf b^2}\sqcup  \widetilde{\mathbf h^1} \sqcup {\mathbf b^1}: \mathbf{c}^1   ]
}{ [  \widetilde{\mathbf b^1}  : \mathbf{c}^0 ] [  {\mathbf b^2}  : \mathbf{c}^2 ]  } 
\in\CC^*/\{\pm 1\}
\end{equation}
Here, for two bases $\mathbf a$ and $\mathbf b$ of a vector space, $[\mathbf a:\mathbf b]$ denotes the determinant the matrix whose colons are the coefficients of the vectors of
$\mathbf a$ as linear combination of $\mathbf b$. 

\begin{Remark}\label{rem:factorTorsion}
The choice of the bilinear form  $\B$ is relevant, as we use a
$\B$-orthonormal basis for  $\mathfrak{g}$ and $\chi(S)\neq 0$. Namely, if we replace $\B$ by $\lambda^2\B$,
then the orthonormal basis will be $\frac1\lambda\{m_1,\ldots,m_{\dG}\}$ and the torsion will be multiplied by
a factor $\lambda^{-\chi(S)\dG}= \lambda^{\dim \R(S,G)} $. 
 \end{Remark}

For an ordered basis $\mathbf a=\{a_1,\ldots, a_m\}$ of a vector space, denote
$$\wedge\mathbf a=a_1\wedge\cdots \wedge a_m\, .$$
Since $\wedge\mathbf a= [\mathbf a:\mathbf b] (\wedge\mathbf b)$, the notation 
$$ 
[\mathbf a:\mathbf b]  = \wedge\mathbf a \,  \slash   {\wedge\mathbf{b}}
$$ is often used in the literature (cf.~\cite{MilnorDuality}).

\subsection{The holomorphic volume form}
 
 The tangent space to $\Rg(S,G) $ at $[\rho]$ is identified to t $H^1(S; \Ad\rho )$, by Proposition~\ref{Prop:tangentspace}.
 There is a natural holomorphic volume form on $H^1(S; \Ad\rho)$:
$$
\Omega_S(\wedge \mathbf h )=\pm \tor(S,\operatorname{Ad}\rho,\mathbf h)
$$
where
$ \mathbf h $ 
is a basis for  $H^1(S;\Ad\rho)$.

The surface $S$ has the simple homotopy type of a graph. Moreover,  graphs that are homotopy equivalenet are also simple-homotopy equivalent, thus this
volume form depends only on the fundamental group
$$
\Omega_{\pi_1(S)}=\Omega_S. 
$$

The bilinear form $\B$ defines 
a bi-invariant volume form $\theta_G$ on the Lie group $G$ in the usual way.
Hence $(\theta_G)^{k}$ is a volume form on $R(\pi_1(S),G)\cong G^k$. 

For a good representation $\rho$ the form $\theta_G$
induces also a form $\theta_{\mathcal{O}(\rho)}$ on the orbit 
$\mathcal{O}(\rho)$ by push-forward:
the orbit map $f_\rho\colon G \to R(\pi_1(S),G)$, $f_\rho(g) = \operatorname{Ad}_g \circ \rho$
factors through the quotient $G/Z(G)$. The quotient map $G\to G/Z(G)$ is a Lie group covering and we get an isomorphism $\bar f_\rho\colon G/Z(G)\to \mathcal{O}(\rho)$, and hence  
\begin{equation}\label{eq:push_forward}
\theta_{\mathcal{O}(\rho)}=
(f_\rho)_* (\theta_{G})\, .
\end{equation}

The next lemma justifies why Reidemeister torsion is the natural choice of volume form
on the variety of representations up to conjugation.

\begin{Lemma}
\label{lemma:volume_orbits}
 Let $\pi\colon R^*( S,G)\to \Rg(S, G)$ denote the projection. 
 Then at $\rho\in R^*( S,G)$ we have:
 $$
 (\theta_G)^{k} = \pm \theta_{\mathcal{O}(\rho)} \wedge \,{\pi^* \Omega_{S}}\, .
 $$
\end{Lemma}

\begin{proof}

We use a graph $\mathcal G$ with one vertex and $k$ edges to compute the torsion of $S$.
 The Reidemeister torsion of this graph is
$
\tor(\mathcal G,\operatorname{Ad}\rho, \mathbf{h})= \pm { [{\mathbf b^1}\sqcup \widetilde{\mathbf{h}}: \mathbf c^1 ]}/{ 
[\widetilde{\mathbf b^1} : \mathbf c^0  ] }.
$
If we make the choice $\widetilde{\mathbf b^1}=\mathbf c^0$, which is a basis for $\mathfrak{g}$, then 
$$
\tor(\mathcal G,\operatorname{Ad}\rho, \mathbf{h})= { [\delta{\mathbf c^0}\sqcup \widetilde{\mathbf{h}}: \mathbf c^1 ]}
= \big(\wedge\delta{\mathbf c^0}\wedge \widetilde{\mathbf{h}}\big) \slash {\wedge\mathbf{c}^1} .
$$
We identify the 1-cells with the generators of $F_k$, so that every element in $\mathbf c^1$ is viewed as a tangent vector to the 
variety of representations, and $\mathbf c^1$ has volume one,
\[
(\theta_G)^{k} ( \wedge \mathbf c^1) = 1
\]
because we started with an $\mathcal{B}$-orthonormal basis for $\mathfrak{g}$. Thus 
\begin{equation}
 \label{eqn:torform}
 (\theta_G)^{k} 
(\wedge\delta{\mathbf c^0}\wedge \widetilde{\mathbf{h}})=\pm 
\tor(\mathcal G,\operatorname{Ad}\rho, \mathbf{h})= \pm  \Omega_S(\wedge \mathbf{h}) .
\end{equation}
As $\delta{\mathbf c^0}$ is a basis of the tangent space to the orbit $\pi_*(\delta{\mathbf c^0})=0$. Moreover, using 
$\pi_*( \widetilde{ \mathbf{h}})= \mathbf{h}$:
\begin{equation}
\label{eqn:ThetaOmega}
(\theta_{\mathcal{O}(\rho)}\wedge \pi^* \Omega_S) (\wedge \delta{\mathbf c^0}\wedge \widetilde{\mathbf{h}})
=
\theta_{\mathcal{O}(\rho)}(\wedge \delta{\mathbf c^0}) \Omega_S(\wedge \mathbf{h}).
\end{equation}
By \eqref{eqn:torform} and \eqref{eqn:ThetaOmega}, to conclude the proof of the lemma we claim that 
$\theta_{\mathcal{O}(\rho)}(\wedge \delta{\mathbf c^0})=1$.
For that purpose,  
we use the canonical identification 
$T_\rho \mathcal{O}(\rho) \cong B^1(\pi_1(S);\Ad\rho)$. 
Using this identification, 
the tangent map of the orbit map $f_\rho\colon G\to\mathcal{O}(\rho)$ at $e\in G$,
$
df_\rho(e)\colon \mathfrak{g} \to T_\rho \mathcal{O}(\rho),
$
corresponds to
$$ 
df_\rho(e)(X) = -\delta(X),
$$ 
where 
$\delta\colon \mathfrak{g}\to B^1(\pi_1(S);\Ad\rho)$ 
denotes the coboundary operator
\[
 \delta(X) (\gamma) = \operatorname{Ad}_{\rho(\gamma)}(X)-X,\qquad\text{ for }
 \gamma\in\pi_1(S).
\]
%
%
%
%
%
%
 Therefore for the basis $ \delta\mathbf{c}^0$ of $B^1(\pi_1(S);\Ad\rho)$
 we obtain by \eqref{eq:push_forward}:
 \begin{equation*}
\theta_{\mathcal{O}(\rho)} (\wedge \delta\mathbf{c}^0) =
\theta_{\mathcal{O}(\rho)} (\wedge df_\rho(e)\mathbf{c}^0) =
\pm \theta_{G}(\wedge \mathbf{c}^0)= 1.
 \end{equation*}
%
%
%
%
%
%
%
%
%
%
%
%
%
%
%
%
%
%
This concludes the proof of  the claim and the lemma.
\end{proof}

 \subsection{The nondegenerate pairing}
 Consider $K'$ the cell decomposition dual to $K$: for each $i$-dimensional cell $e^i_j$ of $K$ there exists a 
dual  $(2-i)$-dimensional cell  $f^{2-i}_j$ of the dual complex $(K',\partial K')$. The complex $C^*(K',\partial K';\mathbb Z)$ yields the relative cohomology of the pair $(S,\partial S)$.
 This can be generalized to cohomology with coefficients.
 If $C_*(\widetilde K;\ZZ)$ denotes the simplicial chain complex on the universal covering, recall from \eqref{eqn:chaincplx} that
 \[
 C^*(K;\Ad\rho)=\hom_{\pi_1(S)}( C_*(\widetilde K;\ZZ), \mathfrak{g}),
\]
and we similarly  define
\[
  C^*(K',\partial K';\Ad\rho)=\hom_{\pi_1(S)}( C_*(\widetilde K',\partial\widetilde K' ;\ZZ), \mathfrak{g})
 \]
where $\pi_1(S)$ acts on $ \mathfrak{g}$ by the adjoint representation.

Following Milnor \cite{MilnorDuality}, there is a paring 
\[
[ .\,,.\,]\colon C_i(\widetilde K;\ZZ)\times C_{2-i}(\widetilde K',\partial\widetilde K' ;\ZZ)\to \ZZ\pi_1(S)
\]
defined by 
\[
[c,c'] := \sum_{\gamma\in\pi_1(S)} (c\cdot \gamma c')\,\gamma .
\]
Here ``$\cdot$'' denotes the intersection number in the universal covering.
The main properties of this paring are that for $\eta\in\ZZ\pi_1(S)$ we have:
\begin{equation}\label{eq:milnor}
[\eta c,c'] = \eta [c,c'], \quad [ c,\eta c'] =  [c,c'] \bar \eta
\quad\text{ and }\quad
[\partial c,c'] = \pm [c,\partial c'] .
\end{equation}
Here  the bar $\overline{ .\,\vphantom{b}}\,\colon\ZZ\pi_1(S)\to \ZZ\pi_1(S) $  denotes the anti-involution  that extends 
$\ZZ$-linearly the anti-morphism of
$\pi_1(S)$ that maps $\gamma\in\pi_1(S)$ to 
$\gamma^{-1}$.
Notice that the sign $\pm$ in equation~\eqref{eq:milnor} depends only on the dimension of the chains.

For each $i$-dimensional cell $e^i_j$ we fix a lift $\tilde e^i_j$ to $\widetilde K$. Also, we chose a
$(2-i)$-dimensional cell $\tilde f^{2-i}_j$ which projects to $f^{2-i}_j$.
By replacing $\tilde f^{2-i}_j$ by a translate, we can assume that
\[
\tilde e^i_j \cdot  \tilde f^{2-i}_k = \delta_{jk} .
\]

We obtain, for each $i$-chain $c\in C_i(\widetilde K ; \ZZ)$ and each
$(2-i)$-chain $c'\in C_{2-i}(\widetilde K',\partial\widetilde K' ; \ZZ)$ that
\[
c = \sum_j [ c, \tilde f^{2-i}_j] \,\tilde e^i_j
\quad\text{ and }\quad
c' = \sum_j \overline{[ \tilde e^i_j, c']}\, \tilde f^{2-i}_j\,.
\]

Given $\alpha\in C^i(K;\Ad\rho) $ and $\alpha'\in C^{2-i-1}(K',\partial K';\Ad\rho)  $ the formula
\[
 (\alpha,\alpha')  \mapsto \sum_{j} 
 \B\big( \alpha(\tilde{e}^i_j ) ,  \alpha'(  \tilde{f}^{2-i}_j)  \big)
\]
defines a  nondegenerate pairing
\begin{equation}
\label{eqn:intersectionchains}
 \langle\cdot,\cdot\rangle\colon C^i(K;\Ad\rho)\times C^{2-i}(K',\partial K';\Ad\rho)  \to  \CC\,  .
\end{equation}

By using equation \eqref{eq:milnor}, it is easy to see that
this pairing satisfies
\begin{equation}
\label{eqn:intersectionboundary}
 \langle \delta \alpha, \alpha'\rangle = \pm \langle  \alpha, \delta \alpha'\rangle\,,
\end{equation}
and therefore it induces a non-singular pairing in cohomology
\begin{equation}\label{eq:pairing-cohom}
\langle\cdot,\cdot\rangle\colon   H^1(S;\Ad\rho)\times H^{1}(S,\partial S;\Ad\rho)  \to  \CC .\\
\end{equation}

Given a basis ${\mathbf{h }}=\{ h_i\}_i $ of $  H^1(S;\Ad\rho) $ and 
${\mathbf{h }}'=\{ h_i'\}_i $ a basis of $  H^1(S,\partial S;\Ad\rho) $, 
 we introduce the notation
\begin{equation}
\langle \wedge{\mathbf{h }}, \wedge{\mathbf{h }} '\rangle := \det\big( \langle h_i, h_j' \rangle_{ij} \big)
\end{equation}
which is the natural extension of the pairing \eqref{eq:pairing-cohom} to 
\[
\bigwedge^d  H^1(S;\Ad\rho) \bigotimes
 \bigwedge^d  H^1(S,\partial S;\Ad\rho) \to\mathbb{C}\,,
 \] 
where $d=-\chi(S)\,\dim G$.

 \subsection{The duality formula}
Let $\rho\in R(\pi_1(S),G)$ be a representation. 
\begin{Proposition}[Duality formula]
\label{Proposition:duality}
Let  ${\mathbf{h }}=\{ h_i\}_i $ be a basis for $  H^1(S;\Ad\rho) $, and let
${\mathbf{h }}'=\{ h_i'\}_i $ be a basis for $  H^1(S,\partial S;\Ad\rho) $. 
Assume that the cohomology groups 
$H^k(S;\Ad\rho)$ and  $H^k(S,\partial S;\Ad\rho)$ 
vanish in dimension $k=0, 2$. Then
\[
 \tor(S,\operatorname{Ad}\rho, {\mathbf{h }}) \, 
 \tor(S,\partial S,\operatorname{Ad}\rho, {\mathbf{h }}') =
 \pm \langle \wedge {\mathbf{h }}, \wedge {\mathbf{h }} ' \rangle
 \]
\end{Proposition}

This is E.~Witten's generalization of the duality formula of W.~Franz and J.~Milnor.
We reproduce the proof for completeness. 
In Witten's article \cite{Witten} the proof of this particular formula is only given 
in the closed case, and Milnor \cite{MilnorDuality} and Franz \cite{Franz37} consider only the acyclic case.

\begin{proof}
 We chose the geometric basis of $C^i(K;\Ad\rho)$ and $C^{2-i}(K',\partial K';\Ad\rho)$
 to be dual to each other, by choosing dual lifts of the cells and 
 a $\B$-orthonormal basis of the Lie algebra $\mathfrak{g}$.
 In this way, the matrix of the intersection form 
 \eqref{eqn:intersectionchains} with respect the geometric basis
 is the identity, in particular its determinant is 1:  $\langle \wedge\mathbf{c}^i , \wedge(\mathbf{c}^{2-i})'\rangle=1$.
 Thus we view the product of torsions in the statement of the proposition as three changes 
 of basis, one for each intersection form:
 \begin{multline}
 \label{eqn:producttorsions}
   \tor(S,\operatorname{Ad}\rho, {\mathbf{h }}) \, 
 \tor(S,\partial S,\operatorname{Ad}\rho, {\mathbf{h }}')\\
= \pm
   \tor(S,\operatorname{Ad}\rho, {\mathbf{h }}) \, 
 \tor(S,\partial S,\operatorname{Ad}\rho, {\mathbf{h }}') \frac{\langle \wedge\mathbf{c}^1 , \wedge(\mathbf{c}^1)'\rangle}{\langle \wedge\mathbf{c}^0 , \wedge(\mathbf{c}^2)'\rangle
 \langle \wedge\mathbf{c}^2 , \wedge(\mathbf{c}^0)'\rangle}
 \\
 = \pm
  \frac{[  \widetilde{\mathbf b^2}\sqcup {\mathbf b^1}\sqcup  \widetilde{\mathbf h} : \mathbf{c}^1   ]\,
  [  \widetilde{(\mathbf b^2)'}\sqcup {(\mathbf b^1)'}\sqcup  \widetilde{\mathbf h'} : (\mathbf{c}^1)'   ]
}{ [  \widetilde{\mathbf b^1}  : \mathbf{c}^0 ] [  {\mathbf b^2}  : \mathbf{c}^2 ] \,
 [  (\widetilde{\mathbf b^1})'  : (\mathbf{c}^0)' ] [  ({\mathbf b^2})'  : (\mathbf{c}^2)' ]}
 \frac{\langle \wedge\mathbf{c}^1 , \wedge(\mathbf{c}^1)'\rangle}{\langle \wedge\mathbf{c}^0 , \wedge(\mathbf{c}^2)'\rangle
 \langle \wedge\mathbf{c}^2 , \wedge(\mathbf{c}^0)'\rangle}
 \\
= \pm \frac{
\langle  \wedge \widetilde{\mathbf b^2} \wedge {\mathbf b^1} \wedge \widetilde{\mathbf h},  
\wedge  \widetilde{(\mathbf b^2)'} \wedge (\mathbf b^1)'\wedge  \widetilde{\mathbf h'}
\rangle
 }{
  \langle \wedge \widetilde{\mathbf b^1} ,\wedge  ({\mathbf b^2})' \rangle 
 \langle \wedge  {\mathbf b^2} , \wedge (\widetilde{\mathbf b^1})' \rangle 
 } .
 \end{multline}
Next, following Witten, we may chose the lift of the coboundaries to be orthogonal 
to the lift of the cohomology of the other complex: 
$$
 \langle   \widetilde{ h_i},   \widetilde{( b^2_j)'} \rangle = 
 \langle  \widetilde{ b^2_i}, \widetilde{ h_j'}\rangle  =0.
$$
%
In addition, by direct application of \eqref{eqn:intersectionboundary}:
$$
\langle   { b^1_i},   ( b^1_j)'\rangle   = 
	   \langle  { b^1_i},  \widetilde{h_j'}\rangle =
	    \langle \widetilde {h_i}, ( b^1_j)' \rangle=0.
$$
Thus the numerator in  \eqref{eqn:producttorsions}  is the determinant of a matrix with some vanishing blocks, 
and \eqref{eqn:producttorsions} becomes:
\begin{equation}
\label{eqn:productsimpler}
\pm \frac{  
  \langle \wedge  \widetilde{\mathbf b^2}, \wedge (\mathbf b^1)'\rangle  
   \langle \wedge  {\mathbf b^1},\wedge  \widetilde{(\mathbf b^2)'}\rangle  
   \langle \wedge {\mathbf h}, \wedge \mathbf h'\rangle 
 }{
  \langle\wedge  \widetilde{\mathbf b^1} ,\wedge  ({\mathbf b^2})' \rangle 
 \langle \wedge  {\mathbf b^2} ,\wedge   (\widetilde{\mathbf b^1})' \rangle 
 }.
\end{equation}
Finally, since $\delta \widetilde{\mathbf b^i}= \mathbf b^i$ and 
 $\delta \widetilde{(\mathbf b^i)'}= (\mathbf b^i)'$, 
 $
  \langle \wedge  \widetilde{\mathbf b^2}, \wedge (\mathbf b^1)'\rangle =\pm \langle \wedge  {\mathbf b^2} ,\wedge  (\widetilde{\mathbf b^1})' \rangle 
 $
and 
$
  \langle \wedge  {\mathbf b^1},\wedge \widetilde{(\mathbf b^2)'} \rangle  =  
  \pm \langle \wedge  \widetilde{\mathbf b^1},\wedge {(\mathbf b^2)'} \rangle 
$, by \eqref{eqn:intersectionboundary}. Hence \eqref{eqn:productsimpler} equals
$\pm \langle  \wedge \mathbf h, \wedge \mathbf h'\rangle $, concluding the proof.
\end{proof}

\begin{Remark}
Notice that the proof generalizes in any dimension, after changing the product by a quotient in the odd dimensional case,
and taking care of the intersection product in all cohomology groups.
\end{Remark}

\section{Symplectic form and volume forms}
\label{Section:forms} 

\subsection{The symplectic form on the relative variety of representations}

For a good and $\partial$-regular representation $\rho_0$, 
the tangent space to $\Rg(S,\partial S,G)_{\rho_0}$ is the kernel of the map $i\colon H^1(S;\Ad\rho) \to H^1(\partial S;\Ad\rho)$
induced by inclusion (Proposition~\ref{prop:relsmooth}). The long exact sequence in cohomology of the pair is:
\begin{equation*}
  0\to   
{H^0(\partial S;\Ad\rho)}\overset\beta\to  
 {H^1(S,\partial S;\Ad\rho)}
 \overset  j\to {H^1(S;\Ad\rho)}\overset i\to  {H^1(\partial S;\Ad\rho)}
  \to 0
\end{equation*}
For $a,b\in \ker(i)$, we define
\begin{equation}
 \label{eqn:omega}
\omega(a,b)= \langle  \widetilde a,b \rangle,
= \langle a, \widetilde b \rangle
 \end{equation}
where $\widetilde a, \widetilde  b\in  H^1(S,\partial S;\Ad\rho) $ satisfy
$j(\widetilde a)=a$, $j(\widetilde  b)=b$.
This form is well defined (independent of the lift), because $i$ and $\beta$ are dual maps with respect to the pairing \eqref{eq:pairing-cohom}, that is 
$\langle \beta(\cdot),\cdot\rangle= \langle \cdot, i(\cdot)\rangle$.
Moreover we have:

\begin{Theorem}[\cite{GoldmanAdvances,GHJW,LawtonP}] 
Assume that $\rho_0$ is a good and $\partial$-regular. Then
the form $\omega$ is symplectic on 
$\Rg(S,\partial S,G)_{\rho_0}$.
 \end{Theorem}

 The fact that $\omega$ is bilinear and alternating is clear from construction, non-degeneracy follows from
 Poincar\'e duality, and the deep result is to prove $d\omega=0$.
 When $S$ is closed this was proved by Goldman in ~\cite{GoldmanAdvances}.
  When $\partial S\neq\emptyset$, the result  with real coefficients is due to  Guruprasad,  Huebschmann,  Jeffrey,  and  Weinstein \cite{GHJW}, and in
 \cite{LawtonP} Lawton explains why it works in the complex case.
 

 \subsection{Sign refined Reidemeister torsion for the circle}
 \label{sec:sgn_ref_tor}
 
Let $V$ be a finite dimensional real or complex vector space, and 
$$
\varphi\colon\pi_1(S^1)\to \mathrm{SL}(V)
$$ 
be a representation. 
In what follows we use the refined torsion with sign due to Turaev, that we denote $\TOR(S^1,\varphi ,\mathfrak{o},\mathbf{u},\mathbf{v})$ 
\cite[\S3]{TuraevKnotTheory}. 
This torsion depends on the choice of an orientation $\mathfrak{o}$ in cohomology with constant coefficients of $S^1$
and the choice of respective basis $\mathbf{u}$ for $H^0(S^1;\varphi)$ and $\mathbf{v}$ for $H^1(S^1;\varphi)$. For a circle $S^1$, 
the choice of an orientation determines 
a fundamental class, hence an orientation in homology. 

We start with a cell decomposition $K$ of $S^1$, with $i$-cells $e^i_1,\ldots, e^i_a$,  $i=0,1$, and a 
(real or complex) basis $\{m_1,\ldots,m_k\}$ for the vector space $V$. The geometric basis 
for $C^i(K;\varphi)$
is then 
$\mathbf{c}^i=\{(\tilde e^i_1)^*\otimes m_1,(\tilde e^i_1)^*\otimes m_2,\ldots, (\tilde e^i_a)^*\otimes m_k \}$. 
As before, let $B^1=\operatorname{Im} (\delta\colon C^0(K;\varphi)\to  C^1(K;\varphi))$ denote the coboundary space and chose 
$\mathbf{b}^1$ as basis for $B^1$ and lift it to  $\widetilde{\mathbf{b}^1}$ in $ C^0(K;\varphi) $. Consider also
$\widetilde {\mathbf{v}}\subset C^1(K;\varphi) $ a representative of $\mathbf{v}$ and similarly 
$\widetilde {\mathbf{u}}\subset C^0(K;\varphi) $ for $\mathbf{u}$.
Then we define the torsion: 
$$
\operatorname{tor}(S^1, \varphi, \mathbf{u},\mathbf{v},\mathbf{c}^0,\mathbf{c}^1)=
\frac{[\widetilde {\mathbf{v}}\sqcup\mathbf{b}^1 : \mathbf{c}^1 ]}
{[ \widetilde{\mathbf{b}^1}\sqcup \widetilde {\mathbf{u}} : \mathbf{c}^0 ]}
\in \CC^*.
$$
Notice that there is no sign indeterminacy, because  we include $ \mathbf{c}^i $ in the notation. In fact
sign indeterminacy comes from changing the order  or 
the orientation of the cells of $K$. The sign is not affected by the choice of a basis for $V$, because $\chi(S^1)=0$. 

Following \cite[\S3]{TuraevKnotTheory} we consider $\alpha_i=\sum_{l=0}^i\dim C^l(K;\varphi)$, 
$\beta_i=\sum_{l=0}^i\dim H^l(S^1;\varphi)$ and $N=\sum_{i\geq 0}\alpha_i\beta_i$. We define
$$
\operatorname{Tor}(S^1, \varphi, \mathbf{u},\mathbf{v},\mathbf{c}^0,\mathbf{c}^1)=
(-1)^N\operatorname{tor}(S^1, \varphi, \mathbf{u},\mathbf{v},\mathbf{c}^0,\mathbf{c}^1).
$$
This quantity is invariant under subdivision of the cells of $K$, but it still depends on their
ordering and orientation. To make this quantity invariant, Turaev introduces the notion of 
cohomology orientation, i.e. an orientation of the $\RR$-vector space $H^0(S^1;\RR)\oplus H^1(S^1;\RR)$.
We consider a geometric basis the complex with trivial coefficients $C^i(K; \RR)$, 
${c}^i=\{(e^i_1)^*,\ldots, (e^i_a)^*\}$, with the same ordering and orientation of cells. 
We chose any basis $h^i$ of $H^i(S^1;\RR)$ that yield the orientation $\mathfrak{o}$.

\begin{Definition}
 The \emph{sign determined} torsion is
 $$
\operatorname{TOR}(S^1,\varphi,\mathfrak{o},\mathbf{u},\mathbf{v})=
\operatorname{Tor}(S^1, \varphi, \mathbf{u},\mathbf{v},\mathbf{c}^0,\mathbf{c}^1)\cdot
\operatorname{sgn} \left(\operatorname{Tor}(S^1, 1, h^0,h^1,{c}^0,{c}^1)\right)^{\dim\varphi}
$$
\end{Definition}

Let $-\mathfrak{o}$ denote the homology orientation opposite to $\mathfrak{o}$. It is straightforward
from construction that
\begin{equation}\label{eq:change_or}
\operatorname{TOR}(S^1,\varphi,-\mathfrak{o},\mathbf{u},\mathbf{v})=
(-1)^{\dim\varphi}\operatorname{TOR}(S^1,\varphi,\mathfrak{o},\mathbf{u},\mathbf{v})
\end{equation}
In particular, we do not need the homology orientation when $\dim\varphi$ is even.
For a circle $S^1$, the choice of an orientation determines 
a fundamental class, hence an orientation in cohomology. 

Let $\varphi_i\colon\pi_1(S^1)\to\mathrm{SL}(V_i)$ be representations into finite dimensional vector spaces, for $i=1,2$.
Then $H^*(S^1;\varphi_1\oplus\varphi_2)\cong H^*(S^1;\varphi_1)\oplus H^*(S^1;\varphi_2)$. Let $\mathbf{u}_i$ and 
$\mathbf{v}_i$ denote basis for $H^0(S^1;\varphi_i)$ and $H^1(S^1;\varphi_i)$ respectively. The following lemma reduces to an elementary calculation:

\begin{Lemma}\label{lem:TOR_mult}
 Let $\varphi_i\colon\pi_1(S^1)\to\mathrm{SL}(V_i)$ be representations into finite dimensional vector spaces,
 for $i=1,2$. Then
\begin{multline*}
 \operatorname{TOR}(S^1,\varphi_1\oplus\varphi_2,\mathfrak{o},\mathbf{u}_1\times\{0\}\sqcup\{0\}\times\mathbf{u}_2,\mathbf{v}_1\times\{0\}\sqcup\{0\}\times\mathbf{v}_2 )
 \\ =\operatorname{TOR}(S^1,\varphi_1,\mathfrak{o},\mathbf{u}_1,\mathbf{v}_1) \cdot
\operatorname{TOR}(S^1,\varphi_2,\mathfrak{o},\mathbf{u}_2,\mathbf{v}_2).
\end{multline*}
\end{Lemma}

\subsection{An holomorphic volume form on $\R^{reg}(S^1,G)$}
\label{sec:peripheral}

As in the introduction we let $G$ denote a simply-connected, semisimple, complex, linear Lie group, $d = \dim G$, and $r =\operatorname{rk} G$.

\begin{Definition}\label{def:regular_rep}
We call a representation $\rho\colon\pi_1(S^1)\to G$ \emph{regular} if the image of the generator 
of $\pi_1(S^1)$ is a regular element ${g}\in G$.
The set of conjugacy classes of regular representations is denoted by $\R^{reg}(S^1,G)$.
\end{Definition}

Let $\rho\colon\pi_1(S^1)\to G$ be regular.  Then $\dim H^0(S^1;\Ad\rho)=r$, because $H^0(S^1;\Ad\rho)\cong\mathfrak{g}^{\Ad \rho}$. 
As the Euler characteristic of $S^1$ vanishes, $\dim H^1(S^1;\Ad\rho)=r $.
 Furthermore, since $G$ is simply connected,  we have  $H^1(S^1;\Ad\rho)\cong T_{[\rho]} \R^{reg}(S^1,G)$ (Corollary~\ref{coro:Steinberg}).

By Poincar\'e duality, the pairing
 $$
\langle \cdot,\cdot \rangle \colon H^0(S^1;\Ad\rho)\times H^1(S^1;\Ad\rho) \to H^1(S^1;\CC )\cong \CC 
 $$
 is non degenerate.

In the next lemma we use the refined torsion with sign due to Turaev (see Section~\ref{sec:sgn_ref_tor}).
By \eqref{eq:change_or}  changing  the orientation of $S^1$ changes the torsion
$\TOR(S^1,\operatorname{Ad}\rho, \mathfrak{o},\mathbf{u},\mathbf{v} )$ by a factor
$(-1)^{\dG}= (-1)^r$, as well as $\langle \wedge\mathbf v,  \wedge\mathbf u \rangle$
by the same factor.

Let $G_\RR$ denote the compact real form of the semisimple complex linear group $G$.
We will assume that the restriction of the nondegenerate symmetric bilinear $G$-invariant  
form $\B$ on the Lie algebra to $\mathfrak{g}_\RR$ is positive definite.
This means that $\B$ is a negative multiple of the Killing form. In what follows we will denote by
$\Ad_\RR\colon G_\RR\to \operatorname{Aut}(\mathfrak{g}_\RR)$ the restriction of $\Ad$ to the real form $G_\RR$.
 \begin{Lemma}
 \label{lemma:independent} 
If $\rho\colon \pi_1(S^1)\to G$ is a regular representation, and if $\mathbf u$ and $\mathbf v$ are bases of 
$H^0(S^1;\Ad\rho)$ and $H^1(S^1;\Ad\rho)$ respectively, then
the product 
\[
\TOR(S^1,\Ad\rho,\mathfrak{o}, \mathbf{u},\mathbf{v} )
 \, \langle \wedge\mathbf v,  \wedge\mathbf u \rangle
 \] 
 is independent of $\mathbf u$.
  \end{Lemma}
 
  \begin{Lemma}
 \label{lemma:positive} 
If $\rho\colon \pi_1(S^1)\to G_\RR$ is a regular representation
and if $\mathbf u$ and $\mathbf v$
are bases of $H^0(S^1;\Ad_\RR\rho) $,
 and $H^1(S^1;\Ad_\RR\rho) $ respectively, then 
\[
  \TOR(S^1,\Ad_\RR\rho, \mathfrak{o},\mathbf{u},\mathbf{v} )
 \, \langle \wedge\mathbf v,  \wedge\mathbf u \rangle >0\, .
\]
 \end{Lemma}

\begin{proof}[Proof of Lemma~\ref{lemma:independent}] Let $\mathbf{u}$ and $\mathbf{u}'$ be bases for $H^0(S^1;\Ad\rho)$, and  $\mathbf{v}$ and $\mathbf{v'}$, for $H^1(S^1;\Ad\rho)$.
We change bases by means of the following formulas:
$$
  \TOR(S^1,\operatorname{Ad}\rho,\mathfrak{o}, \mathbf{u'},\mathbf{v'} )= 
    \TOR(S^1,\operatorname{Ad}\rho, \mathfrak{o}, \mathbf{u},\mathbf{v} )
    \frac{[ \mathbf v': \mathbf  v]}{[  \mathbf u':   \mathbf u]}
$$
and
$$
\langle  \wedge \mathbf v',  \wedge  \mathbf u' \rangle = \langle \wedge \mathbf v,  \wedge \mathbf u \rangle
{[ \mathbf v':  \mathbf v]}{[   \mathbf u':  \mathbf u]}\, .
$$
Hence
\begin{equation}
 \label{eqn:change}
  \TOR(S^1,\operatorname{Ad}\rho, \mathfrak{o}, \mathbf{u'},\mathbf{v'} ) \langle  \wedge \mathbf v',  \wedge  \mathbf u' \rangle
  = 
    \TOR(S^1,\operatorname{Ad}\rho, \mathfrak{o}, \mathbf{u},\mathbf{v} ) \langle \wedge \mathbf v,  \wedge \mathbf u \rangle
[ \mathbf v': \mathbf  v]^2 .
\end{equation}
This proves independence of $\mathbf u$.
\end{proof}

\begin{proof}[Proof of Lemma~\ref{lemma:positive}]
We are assuming that the image of $\rho$ is contained in  
the compact real form, $\rho(\pi_1(S))\subset G_\RR$.
By  \eqref{eqn:change} in  the proof of Lemma~\ref{lemma:independent}, the sign is independent
 of $\mathbf v$. By regularity, $H^0(S^1;\Ad_\RR\rho)\subset \mathfrak{g}_\RR$ is a Cartan subalgebra $\mathfrak h$,
and $\B$ restricted to $\mathfrak h$ is positive definite. Hence we may chose an ${\RR}$-basis of $\mathfrak{g}_{\RR}$ compatible with
the orthogonal decomposition $\mathfrak{g}_{\RR}=\mathfrak{h} \perp \mathfrak{h}^\perp$. 
This is also a decomposition of
$\pi_1(S^1)$-modules, and by Lemma~\ref{lem:TOR_mult} the torsion decomposes accordingly as a product of torsions.

We compute the torsion of  $\mathfrak{h} $ first. Since the adjoint action of $H$ on $\mathfrak{h} $
is trivial, we have natural isomorphisms
\begin{equation}
 \label{eqn:isos}
 H^1(S^1;\mathfrak{h})\cong H^1(S^1,\mathbb{R})\otimes \mathfrak{h}\quad\textrm{ and }\quad
  H^0(S^1;\mathfrak{h})\cong H^0(S^1,\mathbb{R})\otimes \mathfrak{h}.
\end{equation}
We chose a cell decomposition of $S^1$ with a single (positively oriented) cell in each dimension.
In particular, as  the adjoint action of $H$ on $\mathfrak{h} $
is trivial, the boundary operator 
$ \delta\colon C^0(K;\mathfrak{h})\to  C^1(K;\mathfrak{h}) $ vanishes.
Chose a $\B$-orthonormal basis for $\mathfrak{h}$, this provides geometric basis $\mathbf c^1$ and $\mathbf c^0$, and since 
$\delta=0$, those are also representatives of basis in cohomology.
By choosing those bases ($\mathbf u=\mathbf c^0$ and $\mathbf v=\mathbf c^1$), 
$$
\mathrm{tor}(S^1, \Ad\rho\vert_{\mathfrak{h}},  \mathbf c^0,\mathbf c^1,  \mathbf c^0,\mathbf c^1)=1.
$$
Following the construction in Section~\ref{sec:sgn_ref_tor}, we compute $\alpha_0=\beta_0=r$ and 
$\alpha_1=\beta_1=2r\equiv 0 \bmod 2$. Thus $N \equiv r^2\equiv r\bmod 2$ and 
$$
\mathrm{Tor}(S^1, \Ad\rho\vert_{\mathfrak{h}},  \mathbf c^0,\mathbf c^1,  \mathbf c^0,\mathbf c^1)=(-1)^r.
$$
As the torsion for the trivial representation corresponds to the case $r=1$, $\mathrm{Tor}$
for the trivial representation is $-1$ and 
\begin{equation}
\label{eqn:TORCartan}
\TOR(S^1, \Ad\rho\vert_{\mathfrak{h}},\mathfrak{o},  \mathbf c^1,\mathbf c^0)=(-1)^r\cdot\mathrm{sgn}(-1)^r= 1.
\end{equation}
Also, by construction, $\langle \wedge{\mathbf c^1},\wedge{\mathbf c^0}\rangle=1$. 

Next we compute the torsion of  $\mathfrak{h}^\perp $. We have $H^*(S^1;\mathfrak{h}^\perp)=0$ and, since 
$\dim\mathfrak{h}^\perp$ is even,
$$
 \TOR(S^1 ,\operatorname{Ad}\rho\vert_{ \mathfrak{h}^\perp},\mathfrak{o} )
 =\operatorname{tor}(S^1 ,\operatorname{Ad}\rho\vert_{ \mathfrak{h}^\perp},\mathbf{c}^0,\mathbf{c}^1)
 =\det  (\Ad_\RR(g)-\operatorname{Id})\vert_{ \mathfrak{h}^\perp },
$$
where $g\in G$ is the image of a generator of $\pi_1(S^1)$. 
Notice that, as $\dim \mathfrak{h}^\perp$ is even, the sign is independent of the cohomology orientation. 

Let $\Delta_G$ be the Weyl function \cite{GoodWall}. Then
$$
\det  (\operatorname{Ad}(g)-\operatorname{Id})\vert_{ \mathfrak{h}^\perp }= \Delta_G(g) \Delta_G(g^{-1})= \vert \Delta_G(g) \vert ^2>0
$$
(see \cite[(7.47)]{GoodWall} for details). This finishes the proof of the lemma.
\end{proof}

\begin{Definition}
\label{Def:peripheralform}
 Let $\rho\colon\pi_1 (S^1)\to G$ be a regular representation.
 The \emph{form}
  $$
\nu\colon \bigwedge^r H^1(S^1;\Ad\rho)\to \CC 
 $$
is defined by the formula
\begin{equation}
   \label{eqn:peripheral}
\nu (\wedge\mathbf v)=\pm \sqrt{ \TOR(S^1,\operatorname{Ad}\rho, \mathfrak{o} , \mathbf{u},\mathbf{v} )
 \, \langle \wedge\mathbf v,  \wedge\mathbf u \rangle } 
\end{equation}
for any  basis $\mathbf u$ of  $H^1(S^1;\Ad\rho)$. 
(By Lemma~\ref{lemma:independent}, it is independent of $\mathbf u $.)
\end{Definition}

We are interested in understanding $\nu$ as a differential form on $\R^{reg}(S^1, G)$ for $G$ simply connected.
Recall from \S\ref{subsection:Steinberg} that when $G$ is simply connected, the Steinberg map has coordinates the fundamental characters 
$(\sigma_1,\ldots,\sigma_r)\colon G\to\CC^r$.

\begin{Proposition}
 \label{prop:nudsigma}
 For $G$ simply connected, there exists a constant $C\in\CC^*$  and a choice of sign for $\nu$ such that  
$$
\nu=   C\, d\sigma_1\wedge\dots\wedge d\sigma_{r}.
$$ 
\end{Proposition}

\begin{proof}
Using Steinberg's section $s\colon \CC^r\to G^{reg}$ (Theorem~\ref{thm:Steinberg}), consider for each $p\in\CC^r$ the subagebra
$\mathfrak{g}^{\Ad  s(p) }$
of elements fixed by $ {\Ad  s(p)} $.
By  the constant rank theorem this defines an algebraic vector bundle
$$
\mathfrak{g}^{\Ad\circ s} \to E(s)\to \CC^r.
$$
Since algebraic vector bundles over $\CC^r$ are trivial \cite{Quillen,Suslin},
there is a trivialization $\mathbf u=(u_1,\ldots,u_r)\colon \CC^r\to  E(s)$, so that
$\{ u_1(p),\ldots,u_r(p)\}$ is a basis for $\mathfrak{g}^{\Ad  s(p)}$, for each $p\in\CC^r$.
By the identifications, 
$ T_{[s( p)]}\R^ {reg}(S^ 1, G)  \cong H^1(S^1,\Ad s(p))$  (Corollary~\ref{coro:Steinberg}), and the identification
$\mathfrak{g}^{\Ad  s(p)} \cong H^0(S^1,\Ad s(p))  $,
we have  two $(r,0)$-forms on $\CC^r$:
\begin{equation}
 \label{eqn:forms}
  \langle s_*( - ), \wedge \mathbf{u} \rangle \quad\textrm{ and } \quad
 \TOR(S^1,\operatorname{Ad}s, \mathfrak{o}  , \mathbf{u},s_*( - ) ) . 
\end{equation}
We claim that these forms are both algebraic. Assuming the claim,
they are a polynomial multiple of $d z_1\wedge\cdots\wedge d z_r$, for $(z_1,\ldots,z_r)$ the standard coordinate system for $\CC^r$.
Since they vanish nowhere in $\CC^r$, 
both forms in  \eqref{eqn:forms} are a constant multiple of
$d z_1\wedge\cdots\wedge d z_r$. Viewed as 
 as forms on 
$\R^ {reg}(S^ 1, G)$, they are both a constant multiple of $d\sigma_1\wedge\cdots \wedge d\sigma_r$ and the proposition follows,
once we have shown the claim.


To prove that the forms in \eqref{eqn:forms} are algebraic, 
use a CW-decomposition $K$ of $S^1$ with a $1$ and a $0$-cell, so that the groups of cochains
$C^i(K,\Ad s(p))$, for $i=0,1$,  are naturally identified with $\mathfrak{g}$. 
We also have a natural isomorphism
$
 \left( R_{ s(p) ^{-1}}\right)_*\colon T_{s(p)} G\to \mathfrak{g}
$,
which is precisely  the tangent map to righ multiplicatiuon by $s(p) ^{-1}  $.
This identification maps $ s_*(\partial_{z_i})$ at $p\in\CC^n$ to
$$
v_i(p) = \left(R_{ s(p) ^{-1}}\right)_* \left(\frac{\partial s_{\phantom{i}}}{\partial {z_i}}(p) \right)   \in\mathfrak{g},  
$$
which is a map algebraic on $p\in\CC^r$. Hence the intersection product is
$$
\langle s_*(\partial_{z_1}\wedge\cdots\wedge \partial_{z_r} ),\wedge \mathbf{u}  \rangle=
\det (\langle s_*(\partial_{z_i}), u_j\rangle_{ij})= \det(\mathcal{B}( v_i, u_j)_{ij}),
$$
which is polynomial on $p\in\CC^r$.

To show that the torsion is algebraic,  using again triviality of algebraic bundles on $\CC^r$,
 complete $\mathbf u$  to a   section 
of the trivial bundle
$(u_1,\dotsc,u_r,\dotsc, u_{d})\colon \CC^r \to\mathfrak{g} $.
Setting  $\tilde {\mathbf b}^1=\{u_{r+1},\dotsc,u_{d}\}$, then $\mathbf u(p)\sqcup \tilde {\mathbf b}^1(p)$
is a basis for $\mathfrak{g}$, for each $p\in \CC$. We view  $\mathbf u(p)\sqcup \tilde {\mathbf b}^1(p)$ as a basis for $ C^0( K,\Ad s(p))$, so  
that $\mathbf u(p)$ projects to a basis for  $H^0(S^1,\Ad s(p))$, for every $p\in\CC^r$.
Fix $\mathbf c^0=\mathbf c^1$ a basis for $\mathfrak{g}$. By construction:
$$
   \TOR(S^1,\operatorname{Ad}s, \mathfrak{o}  , \mathbf{u},s_*(\partial_{z_1}\wedge\cdots\wedge \partial_{z_r} ) )    = 
   \pm \frac {[\mathbf v \sqcup \partial   \tilde {\mathbf b}^1    : \mathbf c^1]}{[ \mathbf u \sqcup \tilde {\mathbf b}^1  : \mathbf c^0]} ,
$$
where the sign depends on the orientation in homology, but it is constant on $p$. 
Thus this  is a quotient of algebraic polynomial functions on $\CC^r$, but since it is defined everywhere, it is polynomial.
\end{proof}

\subsection{Witten's formula}
 
Let $\rho\colon\pi_1(S)\to G$ be a good $\partial$-regular representation.
Let $\nu_i$ denote the peripheral form of the $i$-th component of $\partial S$ (Definition~\ref{Def:peripheralform}),
and let $\omega$ denote the symplectic form of the relative character variety \eqref{eqn:omega}.
We aim to prove Theorem~\ref{Thm:vol}, namely, that
  $
  \Omega_{\pi_1(S)}=\pm \frac{1}{n!}\omega^{n}\wedge \nu_1\wedge\cdots\wedge\nu_\bS
  $.

\begin{proof}[Proof of Theorem~\ref{Thm:vol}]

We apply the duality formula (Proposition~\ref{Proposition:duality}) 
and the formula of the torsion for the long exact sequence  of the pair, Equation~\eqref{eqn:torsLEX}
below. 
For this purpose 
we discuss the bases in cohomology. Start with 
$\mathbf u$ a basis for $H^0(\partial S;\Ad\rho)$. If $\beta$ denotes the connecting map of the long exact sequence, 
then complete $\beta(\mathbf u)$ to a basis
for $H^1(S,\partial S;\Ad\rho)$: $\beta (\mathbf u) \sqcup \widetilde{\mathbf h} $. 
Next  we chose $\mathbf v$ a basis for $H^1(\partial S;\Ad\rho)$ that we lift to
$\widetilde{\mathbf v }$ by $i$, and if we set $ j( \widetilde{\mathbf h} )=\mathbf h $, 
then  $\mathbf h\sqcup \widetilde{\mathbf v }$ is a basis for
 $H^1(S;\Ad\rho)$ (and  $\mathbf h$ is a basis for $\ker(i)=\operatorname{Im}(j)$). The bases are organized as follows:
\begin{equation}
\label{eqn:LEX}
 0\to \underset { {\mathbf u}} {H^0(\partial S;\Ad\rho)}
 \overset\beta\to\underset { {\beta(\mathbf u)}\sqcup \widetilde{\mathbf h} } {H^1(S,\partial S;\Ad\rho)}
 \overset  j\to \underset{ {\mathbf h}\sqcup \widetilde{ \mathbf v } }{H^1(S;\Ad\rho)}\overset i\to \underset{ \mathbf v}{H^1(\partial S;\Ad\rho)}\to 0
\end{equation}
As the bases have been chosen compatible with the maps of the long exact sequence,  the product formula for the torsion \cite{MilnorBullAMS} gives:
 \begin{equation}
 \label{eqn:torsLEX}
     \operatorname{tor}(S,\operatorname{Ad}\rho, \mathbf h \sqcup \widetilde{\mathbf v} )
  =\pm
  \operatorname{tor}(S,\partial S,\operatorname{Ad}\rho, \beta (\mathbf u) \sqcup \widetilde{\mathbf h})
 \operatorname{tor}(\partial S,\operatorname{Ad}\rho, \mathbf u, \mathbf v).
 \end{equation}
We shall combine \eqref{eqn:torsLEX} with the duality formula  (Proposition~\ref{Proposition:duality}):
\begin{equation}\label{eqn:torsDUAL}
   \operatorname{tor}(S,\operatorname{Ad}\rho, \mathbf h \sqcup \widetilde{\mathbf v} )
    \operatorname{tor}(S,\partial S,\operatorname{Ad}\rho, \beta (\mathbf u) \sqcup \widetilde{\mathbf h}) =\\
    \pm \langle
    \wedge( \mathbf h\sqcup \widetilde{\mathbf v}), \wedge (  \beta (\mathbf u) \sqcup \widetilde{\mathbf h} )
    \rangle\, .
\end{equation}
We next decompose the right hand side in \eqref{eqn:torsDUAL}.   By naturality of the intersection form,
$$
  \langle  {{h_i }},  \beta(  u_j ) \rangle = 
  \langle  i( {{h_i }}),  u_j \rangle = 
  \langle  i( j(\widetilde{{{h_i }}})),  u_j \rangle = 0\, .
 $$
Hence the right hand side in \eqref{eqn:torsDUAL} becomes:
$$
\langle
    \wedge( \mathbf h\sqcup \widetilde{\mathbf v}), \wedge (  \beta (\mathbf u) \sqcup \widetilde{\mathbf h} )
    \rangle=
 \langle \wedge{\mathbf{h }}, \wedge\widetilde{{\mathbf{h }}} \rangle \cdot
 \langle \wedge \widetilde{ \mathbf v}, \wedge \beta( \mathbf u ) \rangle .
$$
Again by naturality
 $$\langle \wedge \widetilde{ \mathbf v}, \wedge \beta( \mathbf u ) \rangle 
  = 
  \langle \wedge i(\widetilde{ \mathbf v}), \wedge \mathbf u \rangle =
  \langle \wedge { \mathbf v},\wedge  \mathbf u \rangle .
 $$
 In addition, by definition
 $$
 \langle \wedge{\mathbf{h }}, \wedge\widetilde{{\mathbf{h }}} \rangle = \omega(\wedge {\mathbf{h }}, \wedge{\mathbf{h }}).
$$
Thus
\begin{equation}
\label{eqn:intersectprod}
\langle \wedge( \mathbf h\sqcup \widetilde{\mathbf v}), \wedge (  \beta (\mathbf u) \sqcup \widetilde{\mathbf h} )  \rangle=
\pm \langle\wedge \mathbf u,\wedge \mathbf v\rangle \, \omega(\wedge {\mathbf{h }}, \wedge{\mathbf{h }}).  
\end{equation}
Hence by \eqref{eqn:torsLEX}, \eqref{eqn:torsDUAL},  and \eqref{eqn:intersectprod}:
\begin{equation*}
 \operatorname{tor}(S,\operatorname{Ad}\rho,  \mathbf h\sqcup \widetilde{\mathbf v} )^2=
 \pm
 \omega( \wedge {\mathbf{h }}, \wedge {\mathbf{h }} )  \operatorname{TOR}(\partial S 
 ,\operatorname{Ad}\rho, \mathbf u, \mathbf v) \langle  \mathbf u, \mathbf v \rangle \, .
\end{equation*}
Notice that on the right hand side we use Turaev's sign refined torsion.
Next we claim that the sign of this formula is $+$ and not $-$.
It suffices to determine the sign in the compact case. Then the formula will follow in the complex case by a connectedness argument
(the variety of characters of a free group is connected and irreducible, 
and $\partial$-regularity and being good are 
Zariski open properties, hence they fail in a set of
real codimension $\geq 2$).

We show that the sign is $+$ in the compact case by showing that all terms are positive.
Since $
 \operatorname{TOR}(\partial S 
 ,\operatorname{Ad}\rho, \mathbf u, \mathbf v) \langle  \mathbf u, \mathbf v \rangle
$
is positive by Lemma~\ref{lemma:positive}, the sign will follow from the equality
\begin{equation}
\label{eqn:Pfaffian}
 \omega( \wedge {\mathbf{h }}, \wedge {\mathbf{h }} ) = \left(
 \frac{1}{n!}\omega^n(\wedge \mathbf{h })
 \right)^2,
\end{equation}
that will also complete the proof of the theorem.

We give self-contained proof of \eqref{eqn:Pfaffian} by completeness. 
By Darboux's theorem there are local coordinates so that 
$$
\omega=d x_1\wedge d x_2+\cdots+ d x_{2 n-1}\wedge d x_{2n}.
$$
Let $A$ be a matrix of size $2n\times 2n$ whose colons are the components of the vectors of $\mathbf{h }$ in this coordinate
system. Then, if $J$ denotes the matrix of the standard symplectic form,
$$
 \omega( \wedge {\mathbf{h }}, \wedge {\mathbf{h }} )=\det\left( \omega( h_i,h_j)_{ij} \right)=\det (A^t J A)= (\det A)^2.
$$
On the other hand 
$
\omega^n= n! \, d x_1\wedge d x_2\wedge\cdots \wedge d x_{2n} 
$, hence
$$
\frac{1}{n!}\omega^n(\wedge \mathbf{h })=\det A
$$
and we are done.   
\end{proof}

\section{Formulas for the group $\SLn $}
\label{Section:SLn} 

If $G=\SLn$ we can give explicit formulas for several volume forms. 

\subsection{The form $\nu$ for $\SLn$}

We know that $\nu$ is a constant multiple of $d\,\sigma_1\wedge\cdots \wedge d\,\sigma_r$ and we shall determine
the constant, completing the proof of Proposition~\ref{prop:Steinberg}.
 Recall that we chose the $\CC$-bilinear form on $\sln$ to be 
$$\mathcal B (X,Y)=-\operatorname{tr}(X\, Y) \qquad \forall X,Y\in\sln.
$$
 In $\SLn$ the invariant functions are the symmetric functions on the spectrum: if the 
 eigenvalues of $A\in\SLn$ are $\lambda_1,\ldots,\lambda_N$, then
 \begin{equation*}
  \sigma_1(A)  = \sum_i\lambda_i, \quad
  \sigma_2(A)  = \sum_{i<j}\lambda_i\lambda_j, \ \ldots , \quad
  \sigma_{N-1}(A) =\sum_i \frac{1}{\lambda_i} \, .
 \end{equation*}
Those symmetric functions are characterized by Cayley-Hamilton theorem:
$$
A^N-\sigma_1(A) A^{N-1}+ \sigma_2(A) A^{N-2}-\cdots +(-1)^{N-1}\sigma_{N-1}(A) A+ (-1)^N\operatorname{Id}=0
\, .
$$
 We identify $R(S^1,\SLn)$ with the group $\SLn$ by mapping a representation to the image of a generator of $\pi_1(S)$,
 so that $\sigma_i$ is a function on $R(S^1,\SLn)$ invariant under conjugation. On the other hand, $\sigma_1,\ldots,
 \sigma_{N-1}$ are the coordinates of the isomorphism:
$$\R(S^1, \SLn )\cong \SLn/\! /\SLn \cong\CC^{N-1}.$$

 \begin{Proposition}
 \label{prop:thetaSL} Let
 $\nu\colon \bigwedge^{N-1} H^1(S^1,\Ad\rho )\to\mathbb{C}$ denote the volume form in Definition~\ref{Def:peripheralform}.
 On $\R(S^1, \SLn )\cong\CC^{N-1}$ 
   $$ 
   \nu = \pm \left(\sqrt{-1}\right)^{\epsilon(N)} \sqrt N\,  d\sigma_1\wedge\cdots\wedge d\sigma_{N-1}\,,
   $$
   where $\epsilon(N)=(N-1)(N+2)/2$.
   \end{Proposition}

By direct application of the proposition,   we get:

   \begin{Corollary}
   \label{cor:dtrace}
On $\R^{reg}(S^1,\SL2)$ 
    $$
    \nu =\pm \sqrt{2}\, d \operatorname{tr}_\gamma 
    $$
    where $\gamma$ is a generator of $\pi_1(S^1)$.
   \end{Corollary}

   \begin{proof}[Proof of Proposition~\ref{prop:thetaSL}]
We identify the variety of representations of the cyclic group $ \pi_1(S^1)$  with
$\SLn$ by considering the image of a generator, that we call $g$.
To simplify, we may assume that $g$ is semisimple, 
by Proposition~\ref{prop:nudsigma}.
After diagonalizing:
    $$
    g=\begin{pmatrix}
       e^{u_1} & 0 & & 0\\
       0 & e^{u_2} & & 0 \\
        & & \ddots & \\
        0&0 & & e^{u_N}
      \end{pmatrix}
    $$
with $u_1+\cdots + u_N=0$ and all $u_i$ are pairwise different mod $2\pi \sqrt{-1}\ZZ$. 
The Cartan algebra $\mathfrak{h}$ is the subalgebra of diagonal matrices.
Since the decomposition 
$\sln=\mathfrak{h}\perp \mathfrak{h}^\perp$ is preserved by the adjoint
action of $g$, the torsion is the corresponding product of torsions, by Lemma~\ref{lem:TOR_mult}.
By looking at the action on non-diagonal entries of $\sln$, the  torsion of the adjoint representation on $\mathfrak{h}^\perp$ is:
$$
\prod_{i\neq j} (e^{u_i-u_j}-1)=\prod_{i\neq j} (e^{u_i}-e^{u_j})=(-1)^{N(N-1)/2} \prod_{i> j}(e^{u_i}-e^{u_j})^2,
$$
which is the product $\Delta_G(g)\Delta_G(g^{-1})$ of Weyl functions \cite[\S7]{GoodWall}. Thus
\begin{equation}
 \label{eqn:thetaOmega}
\nu=\pm \big(\sqrt{-1}\big)^{N(N-1)/2} \prod_{i> j}(e^{u_i}-e^{u_j} ) \, \theta_H\, ,
\end{equation}
where
\begin{equation}
 \label{eqn:Omega}
 \theta_H(\wedge\mathbf v)= \sqrt{\operatorname{TOR}(S^1,\mathfrak{h},\mathfrak{o} ,\wedge\mathbf v,\wedge\mathbf u)
 \langle \wedge\mathbf v,\wedge\mathbf u\rangle } .
\end{equation}


We use coordinates for the Cartan algebra via the entries of the diagonal matrices: 
$\mathfrak{h}\cong \{(u_1,\ldots,u_N)\in \CC^N\mid \sum u_i=0 \}  $,

\begin{Lemma}
 \label{lemmma:wedgewithdui} The form $\theta_H$ is the restriction to 
 $\{(u_1,\ldots,u_N)\in \CC^N\mid \sum u_i=0 \} $ of the form
 $$
\pm \left( \sqrt{-1} \right)^{(N-1)} \frac{1}{\sqrt N} \, \, \sum_{i=1}^N
  (-1)^{N-i}d\, u_1\wedge\cdots\wedge \widehat{d\, u_ i}\wedge \cdots \wedge d\, u_N, $$
or, equivalently, of
 $$
 \pm \left( \sqrt{-1} \right)^{(N-1)} \sqrt N \, d\, u_1\wedge\cdots\wedge d\, u_{N-1}.
 $$
\end{Lemma}

\begin{proof}
In order to compute $\operatorname{TOR}(S^1,\mathfrak{h},\mathfrak{o} ,\wedge\mathbf v,\wedge\mathbf u)$ we proceed as in 
the proof of Lemma~\ref{lemma:positive}. In particular we chose a cell decomposition of $S^1$ with a single 
(positively oriented) cell in each dimension,
and bases in homology represented by the geometric bases. With this choice of $\mathbf{u}$ and $\mathbf{v}$, by
\eqref{eqn:TORCartan}, 
$$
\operatorname{TOR}(S^1,\mathfrak{h}, \mathfrak{o}, \mathbf v,\mathbf u)=1 . 
$$
Next we compute $ \langle \wedge\mathbf v,\wedge\mathbf u\rangle $. The basis $\mathbf{u}$ and $\mathbf{v}$ are constructed from 
dual basis in $H^*(S^1;\ZZ)$ tensorized by a basis of $\mathfrak{h}$.
We choose a basis for the  Cartan subalgebra, $\mathbf e=\{e_1,\ldots, e_{N-1}\}$:
$$
e_1=\left(\begin{smallmatrix}
       1 &  & & \\
        & 0 & &  \\
        & & \ddots & \\
        & & & - 1
      \end{smallmatrix}\right),\quad
e_2=\left(\begin{smallmatrix}
       0 &  & & \\
        & 1 & &  \\
        & & \ddots & \\
        & & & - 1
      \end{smallmatrix}\right),\quad\ldots,\quad 
e_{N-1}=\left(\begin{smallmatrix}
       0 &  & & \\
        & \ddots & &  \\
        & & 1 & \\
        & & & - 1
      \end{smallmatrix}\right) \, .
$$
%
%
Since the cells of $S^1$ are positively oriented,
$$
\langle \wedge \mathbf v, \wedge\mathbf u\rangle= \det( \mathcal B(e_i,e_j))_{i,j}) .
$$
In addition, as  $\mathcal B(e_i,e_i)=-2$ and  $\mathcal B(e_i,e_j)= -1$ for  $i\neq j$,
$\det( \mathcal B(e_i,e_j))_{i,j}) = (-1)^{N-1} N
$.
Thus
\begin{equation}\label{eqn:omegahv}
\theta_H(\wedge \mathbf v)= \pm  \sqrt{ (-1)^{N-1} N  }
 \end{equation}
On the other hand,  direct computation yields: 
$$
\sum_{i=1}^N (-1)^{N-i}d\, u_1\wedge\cdots\wedge \widehat{d\, u_ i}\wedge \cdots\wedge d\, u_N (e_1\wedge\cdots\wedge e_{N-1})=N\, .
$$
By the natural identification of $H^1(S^1;\mathfrak{h})$ with the Cartan algebra $\mathfrak{h}$
we get the lemma.
\end{proof}

We conclude the proof of Proposition~\ref{prop:thetaSL}.
By \eqref{eqn:thetaOmega} and Lemma~\ref{lemmma:wedgewithdui}, 
\begin{equation}
\label{eqn:thetadsum}
\nu\wedge (d\,u_1+\cdots + d\,u_N )=  \big(\sqrt{-1}\big)^{\epsilon(N)}\sqrt{N} \prod_{i> j}(e^{u_i}-e^{u_j} )\,  d\,u_1\wedge \cdots \wedge d\,u_N.
\end{equation}
Next we use Newton's identities:
\begin{eqnarray*}
 \sigma_1 & = & e^{u_1}+ \cdots + e^{u_N} \\
 \sigma_2 & = &-\tfrac{1}{2} (e^{2 u_1}+ \cdots + e^{2 u_N} -\operatorname{Pol}(\sigma_1) ) \\
  &\vdots&  \\
 \sigma_j & = &  (-1)^{j+1}\tfrac{1}{j} (e^{j u_1}+ \cdots + e^{j u_N} -\operatorname{Pol}(\sigma_1,\ldots,\sigma_{j-1}) ) 
\end{eqnarray*}
were $ \operatorname{Pol}(\sigma_1,\ldots,\sigma_{j-1})$ denotes a polynomial expression on $\sigma_1,\ldots,\sigma_{j-1}$,
whose precise value is not relevant here.
From them we deduce
$$
 d\sigma_1\wedge\cdots\wedge d\sigma_{N-1}= \pm \frac{1}{(N-1)!}d (\sum e^{u_j}) \wedge d (\sum e^{2 u_j})\wedge \cdots 
 \wedge d (\sum e^{(N-1)u_j})
$$
Since 
$$
d (\sum e^{i u_j}) = i  \sum e^{i u_j}d\, u_ j, 
$$
for $i=1,\ldots , m-1$, Vandermonde determinant yields
\begin{equation}
 \label{eqn:Vandermonde}
(d\,u_1+\cdots + d\,u_N ) \wedge d\sigma_1\wedge\cdots\wedge d\sigma_{N-1} =
\pm \prod_{i>j} (e^{u_i}-e^{u_j}) d\,u_1\wedge \cdots \wedge d\,u_N.
 \end{equation}
Then combine \eqref{eqn:thetadsum} and \eqref{eqn:Vandermonde} to prove the theorem, knowing that our tangent space is 
the kernel of $ d\,u_1+\cdots + d\,u_N $.
\end{proof}


\subsection{The form $\nu$ for $\SUn$}

An element in  $\SUn$ is conjugate to a diagonal element
$$
\begin{pmatrix}
 e^{i\theta_1} & & \\
 & \ddots &  \\
 & & e^{i\theta_N}
\end{pmatrix}
$$
with $\sum\theta_i\in 2\pi\ZZ$.
A matrix is regular if and only if $e^{i\theta_j}\neq e^{i\theta_k}$
for $j\neq k$.

By identifying $\R^{reg}(S^1, \SUn)$ with the image of the generator (or its conjugacy class),
functions on 
$\theta_1,\ldots,\theta_N$ invariant under permutations are well defined on $\R^{reg}(S^1, \SUn)$.
Also the form $d\,\theta_1\wedge\cdots\wedge d\, \theta_{N-1}$ is well defined up to sign  by the relation  
$\sum\theta_i\in 2\pi\ZZ$.
 
\begin{Proposition}
On $\R^{reg}(S^1, \SUn)$  (for $\mathcal B(X,Y)=-\operatorname{tr}(X\, Y)$),
$$
\nu= \pm 2^{N(N-1)/2} \sqrt{N} \prod_{i<j} \sin\left(\frac{\theta_i-\theta_j}{2}\right)  d\,\theta_1\wedge\cdots\wedge d\, \theta_{N-1} .
$$
\end{Proposition}

\begin{proof}
 From the proof of  Proposition~\ref{prop:thetaSL}, if $g\in \SUn $ is the image of the generator of $\pi_1(S^1)$,
 $\nu= \vert \Delta_{G}({g})\vert \,\theta_{ H}
 $.
 By \cite[Exercise 7.3.5]{GoodWall}, 
 $$
\vert \Delta_{G}({g}) \vert = 2^{N(N-1)/2} \prod_{i<j} \left\vert \sin\left(\frac{\theta_i-\theta_j}{2}\right) \right\vert .
 $$
 On the other hand, by Lemma~\ref{lemmma:wedgewithdui},
 $$
 \theta_{ H}= \pm\sqrt{N} \,   d\,\theta_1\wedge\cdots\wedge d\, \theta_{N-1},
 $$
 which proves the formula.
\end{proof}

\begin{Remark}
We may consider also the restriction to $\operatorname{SL}_N(\RR)$. Then the expression of the volume form is just the restriction of Proposition~\ref{prop:thetaSL}.
It may be either real valued or  $\sqrt{-1}$ times real,  
because $\mathcal{B}$ is not positive definite on $\mathfrak{sl}_N(\RR)$. The restriction of $\mathcal{B}$ to $\mathfrak{so}_N$ is positive definite, but its 
restriction to its orthogonal 
$\mathfrak{so}_N^\perp\subset \mathfrak{sl}_N(\RR)$
is 
negative definite. Notice that $\dim \mathfrak{so}_N^\perp =(N-1)(N+2)/2\equiv \epsilon (N) \mod 2$,  that determines whether it is real or $\sqrt{-1}$ times real. 
\end{Remark}

\subsection{Volume form for representation spaces of  $F_2$}
 
In this subsection we compute the volume  form on  the space of 
representations of a free group of rank 2, $F_2=\langle \gamma_1,\gamma_2\rangle$,  in $\SL2$ and $\SLL_3(\CC)$.
We use the notation $t_{i_1\cdots i_k}$
for the trace functions of $\gamma_{i_1}\cdots \gamma_{i_k}$ in $\SL2$, with the convention 
$\gamma_{\bar i}=\gamma_i ^{-1}$.
For instance, the trace function of $\gamma_1\gamma_2^{-1}$
will be denoted by $t_{1\bar 2}$.
 
We start with $\R(F_2,\SL2)$. 
By Fricke-Klein theorem, see \cite{Goldman09}, the
respective trace functions
of $\gamma_1$, $\gamma_2$ and $\gamma_{1}\gamma_2$ define an isomorphism
\begin{equation}
 \label{eqn:Fricke}
 (t_1,t_2,t_{12})\colon X(F_2,\SL2)\to \CC^3\, .
\end{equation}
Since $F_2$ is the fundamental group of a pair of pants $S_{0,3} $, and  $\gamma_1$, $\gamma_2$ and $\gamma_{1}\gamma_2$
correspond to the peripheral elements, by Theorem~\ref{Thm:vol} and Corollary~\ref{cor:dtrace}:

\begin{Corollary}
 \label{cor:C3}
 The volume form on $\R^*(F_2,\SL2)$
  is
 $$
\Omega_{ F_2}= \pm 2\sqrt{2}\,  d\, t_1 \wedge d\, t_2 \wedge d\, t_{12}\, . 
$$
\end{Corollary}

\medskip

We next discuss the space of representations of $F_2=\langle \gamma_1,\gamma_2\rangle$ in  $\SLL_3(\CC)$. 


The symmetric invariant functions $\sigma_1$ and $\sigma_2$ of a matrix in $\SLL_3(\CC)$  are, respectively, its trace and the trace of its inverse. 
Recall that the trace functions in $\SLL_3(\CC)$ are denoted by $\uptau_{i_1\cdots i_k}$ instead of $t_{i_1\cdots i_k}$.
According to \cite{LawtonJA}, 
$X(F_3,\SLL_3(\CC))$ is a branched covering of $\CC^8$ with coordinates
$$
\mathcal{T}=(\uptau_1,\uptau_{\bar 1},\uptau_2,\uptau_{\bar 2},\uptau_{12},\uptau_{\bar 1\bar 2},\uptau_{1\bar 2}, \uptau_{\bar 12})
\colon X(F_2, \SLL_3(\CC))\to \CC^8.$$
The branching is given by the trace of the commutators $\uptau_{12\bar1\bar2}$ and $\uptau_{21\bar 2\bar 1}$
that are solutions of a quadratic equation
$$
z^2 - P z + Q = 0
$$
for some polynomials $P$ and $Q$ on the variables $\uptau_1,\uptau_{\bar 1},\uptau_2,\uptau_{\bar 2},\uptau_{12},\uptau_{\bar 1\bar 2},\uptau_{1\bar 2}, \uptau_{\bar 12}$,
the expression of $P$ and $Q$ can be found in \cite{LawtonJA,LawtonP}.
Notice that $P= \uptau_ {12\bar1\bar2} + \uptau_{21\bar 2\bar 1}$ and
$Q = \uptau_ {12\bar1\bar2} \, \uptau_{21\bar 2\bar 1}$.

Thus,  as   $\gamma_1$, $\gamma_2$ and $\gamma_{1}\gamma_2$ represent
the peripheral elements  of a pair of pants $S_{0,3} $,
a generic subset of the relative variety of representations is locally parameterized by $(\uptau_{1\bar 2}, \uptau_{\bar 1 2})$; in the
subset of points where there is no branching, i.e.~$\uptau_ {12\bar1\bar2} \neq \uptau_{21\bar 2\bar 1}$.
Lawton has computed in \cite[Thm.~25]{LawtonP} the Poisson bracket:
$$
\{\uptau_{1\bar 2}, \uptau_{\bar 1 2}  \}= \uptau_{21\bar 2\bar 1}-\uptau_{12\bar1\bar2}.
$$
As $(\uptau_{1\bar 2}, \uptau_{\bar 1 2})$ are local coordinates, an elementary computation   yields
\begin{equation}     
 \label{eqn:poisson}
 \omega=-\frac{1}{ \{\uptau_{1\bar 2}, \uptau_{\bar 1 2}  \} } d\, \uptau_{1\bar 2}\wedge d\, \uptau_{\bar 1 2}.
\end{equation}
Therefore
\begin{equation}
\label{eqn:symplecticSL3}
\omega= \frac{d\, \uptau_{1\bar 2}\wedge d\, \uptau_{\bar 1 2}  }{  \uptau_{21\bar2\bar1}-\uptau_{12\bar1\bar2}}.
\end{equation}
On the other hand, by  Proposition~\ref{prop:thetaSL}, the form $\nu_1$ corresponding to $\gamma_1$ is
$$
\nu_1=\pm\sqrt{-3} \, d\, \uptau_1\wedge d\, \uptau_{\bar 1},
$$
and similarly for $\gamma_2$ and $\gamma_{12}$. 
Using Theorem~\ref{Thm:vol} and these computations we get:

\begin{Proposition}
\label{Prop:volF2SL3}
For 
\[
\mathcal{T}=(\uptau_1,\uptau_{\bar 1},\uptau_2,\uptau_{\bar 2},\uptau_{12},\uptau_{\bar 1\bar 2},\uptau_{1\bar 2}, \uptau_{\bar 12})
\colon \R^*(F_2,\SLL_3(\CC) )\setminus 
\{   \uptau_{21\bar2\bar1} =  \uptau_{12\bar1\bar2}\}\to \CC^8 
\]
the restriction of the holomorphic volume form
on
$\R^*(F_2,\SLL_3(\CC) )\setminus 
\{   \uptau_{21\bar2\bar1} =  \uptau_{12\bar1\bar2}\}$ is $ \Omega_{F_2}^{\SLL_3(\CC)  }= \pm \mathcal{T}^*\Omega$ where
 $$
\Omega=\pm \frac{3\sqrt{-3}}{  \uptau_{21\bar2\bar1}-\uptau_{12\bar1\bar2}}
d\, \uptau_1\wedge d\, \uptau_{\bar 1} \wedge d\, \uptau_2\wedge d\, \uptau_{\bar 2}
\wedge d\, \uptau_{12}\wedge d\, \uptau_{\bar 1\bar 2}\wedge d\, \uptau_{1\bar 2}\wedge d\, \uptau_{\bar 1 2}.
$$ 
\end{Proposition}

\section{Symplectic forms }
\label{Section:symplectic}

Let $\rho_0 \in R^*(S, \SL2)$ be a good, $\partial$-regular representation.
In this section we discuss the symplectic from on the relative character variety 
$\R^*(S,\partial S, \SL2)_{\rho_0}$
for the two surfaces $S_{1,1}$ and $S_{0,4}$, which are the surfaces with 2-dimensional
relative character variety $\R^*(S,\partial S, \SL2)_{\rho_0}$.
We use Goldman's \emph{product formula} for the Poisson bracket 
for surfaces \cite{GoldmanHamiltonian}, as well as  Lawton's generalization \cite[Sec.~4]{LawtonP} to the relative character variety.

For this purpose, let $f\colon G\to\CC$ be an \emph{invariant function} (i.e. a function on $G$ invariant under conjugation).
Following Goldman~\cite{GoldmanComplex}, its \emph{variation function} (relative to $\mathcal{B}$) is defined as the unique map 
$F\colon G\to\mathfrak{g}$ such that for all $X\in\mathfrak{g}$, $A\in G$,
\begin{equation}\label{eqn:variation}
\frac{d}{dt} f\big(A \, \exp(tX)\big) \big|_{t=0} =\mathcal{B}\big(F(A),X\big)\,.
\end{equation}
When $G=\SL2$ and $f=\operatorname{tr}$, the corresponding variation formula $\operatorname{\mathbf T}\colon\SL2\to\sl2$  must satisfy,
by \eqref{eqn:variation},
$\operatorname{tr}(A \, X) = - \operatorname{tr}\big(\operatorname{\mathbf T}(A)\,X\big)$, $\forall X\in\sl2$ and $ \forall A\in\SL2$.
Thus
$$
\operatorname{\mathbf T}(A)=\tfrac{\operatorname{tr} A }{2} \operatorname {Id} -A
=-\frac{  1}{2}(A-A^{-1}) \quad \text{ for $A \in \SL2$.}
$$
Notice that $\operatorname{\mathbf T} (A)\in  \sl2$ is invariant by the adjoint action of $A$, and 
$\operatorname{\mathbf T}(A)\neq 0$ for $A\neq\pm\operatorname{Id}$.

\begin{Proposition}[\cite{GoldmanComplex,LawtonP}]
\label{Prop:GoldmanPoisson}
Let $\alpha,\beta$  be oriented, simple closed curves meeting transversally in double points $p_1,\ldots,p_k\in S$. For $[\rho]\in\R^*(S,\partial S, \SL2)_{\rho_0}$ and each $p_i$, chose representatives
$$
\rho_i\colon\pi_1(S,p_i)\to  \SL2 
$$
of $[\rho]$.
Let $\alpha_i,\beta_i$ be elements in $\pi_1(S,p_i)$ representing  $\alpha,\beta$ respectively.
For the bilinear form 
$\mathcal B(X,Y)= -\operatorname{tr} (X\, Y)$,
the Poisson bracket of the trace functions $t_\alpha$ and $t_\beta$ is
\begin{align*}
\{t_\alpha, t_\beta \}([\rho])
&=
\sum_{i=1}^k \epsilon(p_i,\alpha,\beta)\,
\mathcal{B}\big(\operatorname{\mathbf T}(\rho_i(\alpha_i))  , \operatorname{\mathbf T}(\rho_i(\beta_i))\big) \\
&=-
\sum_{i=1}^k \epsilon(p_i,\alpha,\beta)\,
\operatorname{tr}
\big(\operatorname{\mathbf T}(\rho_i(\alpha_i))  \, \operatorname{\mathbf T}(\rho_i(\beta_i))\big)
\end{align*}
where  $\epsilon(p_i,\alpha,\beta)$ denotes the oriented intersection number of $\alpha$ and $\beta$ at $p_i$. 
\end{Proposition}

 For later computations, it is useful to recall (cf.~\cite{AcunaMontesinos}) that for all $A,B\in \SL2$
\begin{equation} \label{eqn:trace}
\operatorname{tr}(A) \operatorname{tr}(B)=\operatorname{tr}(A B) +\operatorname{tr}(A B^{-1})\,,
\end{equation}
and a direct calculation gives
\begin{equation} \label{eqn:traceInvariant}
 \operatorname{tr}(\operatorname{\mathbf T}(A)\, \operatorname{\mathbf T}(B))=
 \frac{1}{2} \operatorname{tr}(AB-AB^{-1} )\,.
\end{equation}

\subsection{A torus minus a disc}
Let $S_{1,1}$ denote   a surface of genus 1 with a boundary component. Its fundamental group is  freely generated by two elements
$\gamma_1$ and $\gamma_2$ that are represented by curves that intersect at one point. The peripheral element is the commutator 
$[\gamma_1,\gamma_2]=\gamma_1\gamma_2\gamma_1^{-1}\gamma_2^{-1}$.
The variety of characters  $X(S_{1,1},\SL2)$ is the variety of characters of the free group on two generators, and it is isomorphic
to $\CC^3$ with coordinates $t_1,t_2,t_{12}$, by Fricke-Klein  \eqref{eqn:Fricke}.
Equality~\eqref{eqn:trace}
implies that $t_1t_2 = t_{12}+t_{1\bar2}$.

Generically, the relative character variety is the hypersurface of $\CC^3$ that is a level set of the trace of the commutator,
$t_{12\bar 1\bar2}= c$ for some $c\in\CC$, where 
\begin{equation}
 \label{eqn:commutator}
t_{12\bar 1\bar2} =t_1^2+t_2^2+t_{12}^2 - t_1t_2t_{12}-2 \, .
 \end{equation}
Therefore, given a good representation $\rho_0$ the variables $(t_1,t_2)$ define local coordinates
of  $\R^*(S_{1,1},\partial S_{1,1}, \SL2)_{\rho_0}$
precisely when $\frac{\partial \phantom {t_{12}}}{\partial {t_{12}} } t_{12\bar 1\bar2} \neq 0$, ie.\
when
\begin{equation}
 \label{eqn:t1t2coordinates}
 2 t_{12}-t_1 t_2=t_{12}-t_{1\bar 2}\neq 0,
\end{equation}
where $t_{1\bar 2}=t_{\bar 1 2}$ is the trace function of $\gamma_1\gamma_2^{-1}$. Hence we obtain a local parametrization
 \[
 T=(t_1,t_2)\colon \R^*(S_{1,1},\partial S_{1,1}, \SL2)_{\rho_0}\setminus\{t_{12}= t_{1\bar 2}\}\to\CC^2 .
 \] 
We compute next the symplectic form. 

\begin{Proposition}
\label{Prop:symplectic11}
Let $\rho_0 \in R^*(S, \SL2)$ be a good, $\partial$-regular representation
such that $t_{12}(\rho_0)\neq t_{1\bar 2}(\rho_0)$. 
Then
the symplectic form  on
$\R^*(S_{1,1},\partial S_{1,1}, \SL2)_{\rho_0}\setminus\{t_{12}= t_{1\bar 2}\}$ is the pull-back  $T^*\omega$, where  
$$
\omega=\pm 2 \frac{d\, t_1\wedge d\, t_2}{t_{12}- t_{1\bar 2}}.
$$
\end{Proposition}

\begin{proof}
For $[\rho]\in\R^*(S_{1,1},\partial S_{1,1}, \SL2)_{\rho_0}\setminus\{t_{12}= t_{1\bar 2}\}$ we put
$A=\rho(\gamma_1)$ and $B=\rho(\gamma_2)$. As $\gamma_1$ and $\gamma_2$ intersect in a single point,
by Proposition~\ref{Prop:GoldmanPoisson} and \eqref{eqn:traceInvariant}
 the Poisson bracket, for $\mathcal B(X,Y)= -\operatorname{tr} (X\, Y)$,
 between trace functions is 
$$
\{t_1,t_2\}([\rho])=\pm \operatorname{tr}(\operatorname{\mathbf T}(A)\, \operatorname{\mathbf T}(B))
=\pm\frac12(t_{12}-t_{1\bar 2}).
$$
The proposition follows from Equation~\eqref{eqn:poisson}.
\end{proof}

\begin{Remark}
  From Proposition~\ref{Prop:symplectic11} we can compute again the volume form on $\R^*(F_2,\SL2)$,
already found in Corollary~\ref{cor:C3}.
 Namely, by   Theorem~\ref{Thm:vol}, Proposition~\ref{Prop:symplectic11}, and Corollary~\ref{cor:dtrace},
 since the commutator $\gamma_1\gamma_2\gamma_1^{-1}\gamma_2^{-1}$ is the peripheral element,
\begin{equation}
 \label{eqn:OmegaS11}
\Omega_{ F_2 } =\Omega_{ S_{1,1} } =\pm 2\sqrt{2}\,\frac{d\, t_1\wedge d\, t_2\wedge d\, t_{12\bar 1\bar 2}}{   t_{12}-t_{1 \bar 2}  }  . 
\end{equation}
  Differentiating \eqref{eqn:commutator}, we get 
\begin{equation}
 \label{eqn:dcommutator}
d\, t_{12\bar 1\bar2} = ( 2 t_1-t_2t_{12}) d\, t_1 + ( 2 t_2-t_1t_{12}) d\, t_2 + ( 2 t_{12}-t_1t_{2}) d\, t_{12}\, ,   
\end{equation}
thus, as $t_1t_2=t_{12}+t_{1\bar 2}$, by replacing \eqref{eqn:dcommutator} in \eqref{eqn:OmegaS11}:
 $$
\Omega_{ F_2 } =\pm 2\sqrt{2}\, d\, t_1\wedge d\, t_2\wedge d\, t_{12}  .
$$
\end{Remark}

\subsection{A planar surface with four boundary components}
\label{sec:S04}

Let  $S_{0,4}$ denote the planar surface with four boundary components and let  $\lambda$ and $\mu$ be two simple 
closed curves so that each one divides $S_{0,4}$ in two pairs of pants and they intersect  in precisely two points.
Chose also one of the intersections points as
a base point for the fundamental group.

Orient the curves $\lambda$ and $\mu$ and obtain two new oriented curves $\alpha$ and $\beta$, by changing both
crossings in a way compatible with the orientation, according to Figures~\ref{Fig:ab} and \ref{Fig:ba}.

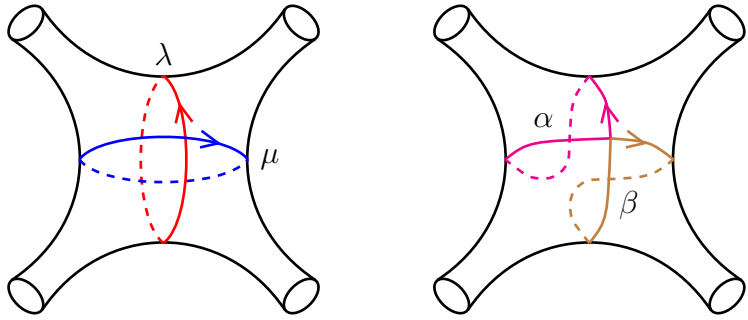
\begin{figure}[h]
\begin{center}
\begin{tikzpicture}[scale=.4]

\begin{scope}[shift={(-14,0)}]
\draw[line width=1pt] (-4, 5) .. controls (-2,2) and (2,2) ..  (4, 5);
\draw[line width=1pt] (5, -4) .. controls (2,-2) and (2,2) ..  (5,4);
\draw[line width=1pt] (-4, -5) .. controls (-2,-2) and (2,-2) ..  (4, -5);
\draw[line width=1pt] (-5, -4) .. controls (-2,-2) and (-2,2) ..  (-5,4);
\draw[line width=1pt,rotate=45]  (6,0) arc (-180:180:.4 and .7);
\draw[line width=1pt,rotate=135]  (6,0) arc (-180:180:.4 and .7);
\draw[line width=1pt,rotate=225]  (6,0) arc (-180:180:.4 and .7);
\draw[line width=1pt,rotate=-45]  (6,0) arc (-180:180:.4 and .7);
   \draw[line width=1pt, red, dashed] (0, -2.77) .. controls (-1,-2) and  (-1,2) ..  (0, 2.77);
  \draw[line width=1pt, blue, dashed] ( -2.77, 0) .. controls (-2,-1) and (2,-1) ..  (2.77,0);
 \draw[line width=1pt, red] (0, -2.77) .. controls (1,-2) and  (1,2) ..  (0, 2.77);
  \draw[line width=1pt, blue] ( -2.77, 0) .. controls (-2,1) and (2,1) ..  (2.77,0);
 \draw[line width=1pt, red] (0.4,1.2)--(0.54,1.8)--(0.95,1.25);
  \draw[line width=1pt, blue] (1.2,0.4)--(1.8,0.54)--(1.25,0.95);
\node at (3.5,0) {$\mu$};
\node at (0,3.5) {$\lambda$};
\end{scope}
\draw[line width=1pt] (-4, 5) .. controls (-2,2) and (2,2) ..  (4, 5);
\draw[line width=1pt] (5, -4) .. controls (2,-2) and (2,2) ..  (5,4);
\draw[line width=1pt] (-4, -5) .. controls (-2,-2) and (2,-2) ..  (4, -5);
\draw[line width=1pt] (-5, -4) .. controls (-2,-2) and (-2,2) ..  (-5,4);
\draw[line width=1pt,rotate=45]  (6,0) arc (-180:180:.4 and .7);
\draw[line width=1pt,rotate=135]  (6,0) arc (-180:180:.4 and .7);
\draw[line width=1pt,rotate=225]  (6,0) arc (-180:180:.4 and .7);
\draw[line width=1pt,rotate=-45]  (6,0) arc (-180:180:.4 and .7);
 \draw[line width=1pt, magenta] (0.7,0.7) .. controls  (0.6 , 2)  ..  (0, 2.77);
 \draw[line width=1pt, magenta] (0.7,0.7) .. controls   (-2, 0.6)  .. ( -2.77, 0);
 \draw[line width=1pt, magenta, rounded corners=10pt, dashed]    (0, 2.77) .. controls (-0.6 , 2) .. (-0.7,-0.7) .. controls   (-2, -0.6)  .. ( -2.77, 0)  ;
 \draw[line width=1pt, magenta] (0.4,1.2)--(0.54,1.8)--(0.95,1.25);
\draw[line width=1pt, brown] (0.7,0.7) .. controls   (2,.6) ..  (2.77,0);
\draw[line width=1pt, brown, rounded corners=10pt, dashed] (0, -2.77) ..  controls (-.6,-2)  ..   (-0.7,-0.7) .. controls   (2,-.6) ..  (2.77,0);
\draw[line width=1pt, brown] (0.7,0.7) .. controls   (.6,-2) ..  (0,-2.77);
  \draw[line width=1pt, brown] (1.2,0.4)--(1.8,0.54)--(1.25,0.95);
\node at (-1.5,1.3) {$\alpha$} ;
\node at (1.3,-1.5) {$\beta$} ;
 \end{tikzpicture}
\end{center}
 \caption{Construction of $\alpha$ and $\beta$ from an orientation of $\lambda$ and $\mu$}
\label{Fig:ab}
\end{figure}

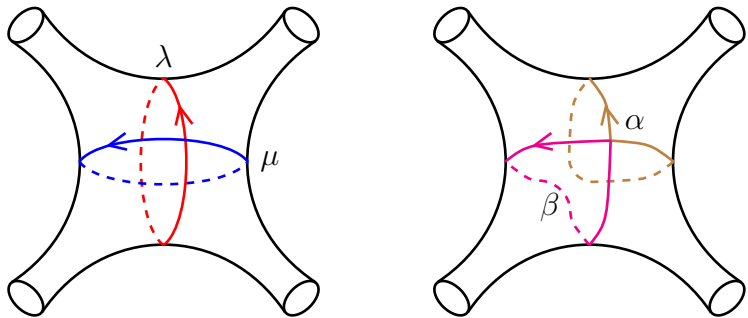
\begin{figure}[h]
\begin{center}
\begin{tikzpicture}[scale=.4]

\begin{scope}[shift={(-14,0)}]
\draw[line width=1pt] (-4, 5) .. controls (-2,2) and (2,2) ..  (4, 5);
\draw[line width=1pt] (5, -4) .. controls (2,-2) and (2,2) ..  (5,4);
\draw[line width=1pt] (-4, -5) .. controls (-2,-2) and (2,-2) ..  (4, -5);
\draw[line width=1pt] (-5, -4) .. controls (-2,-2) and (-2,2) ..  (-5,4);
\draw[line width=1pt,rotate=45]  (6,0) arc (-180:180:.4 and .7);
\draw[line width=1pt,rotate=135]  (6,0) arc (-180:180:.4 and .7);
\draw[line width=1pt,rotate=225]  (6,0) arc (-180:180:.4 and .7);
\draw[line width=1pt,rotate=-45]  (6,0) arc (-180:180:.4 and .7);
   \draw[line width=1pt, red, dashed] (0, -2.77) .. controls (-1,-2) and  (-1,2) ..  (0, 2.77);
  \draw[line width=1pt, blue, dashed] ( -2.77, 0) .. controls (-2,-1) and (2,-1) ..  (2.77,0);
 \draw[line width=1pt, red] (0, -2.77) .. controls (1,-2) and  (1,2) ..  (0, 2.77);
  \draw[line width=1pt, blue] ( -2.77, 0) .. controls (-2,1) and (2,1) ..  (2.77,0);
 \draw[line width=1pt, red] (0.4,1.2)--(0.54,1.8)--(0.95,1.25);
  \draw[line width=1pt, blue] (-1.2,0.4)--(-1.8,0.54)--(-1.25,0.95);
\node at (3.5,0) {$\mu$};
\node at (0,3.5) {$\lambda$};
\end{scope}
\draw[line width=1pt] (-4, 5) .. controls (-2,2) and (2,2) ..  (4, 5);
\draw[line width=1pt] (5, -4) .. controls (2,-2) and (2,2) ..  (5,4);
\draw[line width=1pt] (-4, -5) .. controls (-2,-2) and (2,-2) ..  (4, -5);
\draw[line width=1pt] (-5, -4) .. controls (-2,-2) and (-2,2) ..  (-5,4);
\draw[line width=1pt,rotate=45]  (6,0) arc (-180:180:.4 and .7);
\draw[line width=1pt,rotate=135]  (6,0) arc (-180:180:.4 and .7);
\draw[line width=1pt,rotate=225]  (6,0) arc (-180:180:.4 and .7);
\draw[line width=1pt,rotate=-45]  (6,0) arc (-180:180:.4 and .7);
\draw[line width=1pt, brown] (0.7,0.7) .. controls  (0.6 , 2)  ..  (0, 2.77);
 \draw[line width=1pt, brown, rounded corners=10pt, dashed]   (2.77,0)  .. controls   (2,-.6) ..   (-0.7,-0.7) .. controls (-0.6 , 2)  ..  (0, 2.77);
  \draw[line width=1pt, brown] (0.4,1.2)--(0.54,1.8)--(0.95,1.25);
  \draw[line width=1pt, brown] (0.7,0.7) .. controls   (2,.6) ..  (2.77,0);
 \draw[line width=1pt, magenta] (0.7,0.7) .. controls   (-2, 0.6)  .. ( -2.77, 0);
 \draw[line width=1pt, magenta, rounded corners=10pt, dashed]    (0, -2.77) .. controls (-.6,-2) ..   (-0.7,-0.7) .. controls  (-2, -0.6)  .. ( -2.77, 0);
\draw[line width=1pt, magenta] (0.7,0.7) .. controls   (.6,-2) ..  (0,-2.77);
  \draw[line width=1pt, magenta] (-1.2,0.4)--(-1.8,0.54)--(-1.25,0.95);
\node at (1.5,1.3) {$\alpha$} ;
\node at (-1.3,-1.5) {$\beta$} ;
 \end{tikzpicture}
\end{center}
\caption{Construction of $\alpha$ and $\beta$ from another orientation of $\lambda$ and $\mu$}
\label{Fig:ba}
\end{figure}

Since the curves are oriented, we may talk about the elements they represent in $\pi_1(S_{0,4})$,
in particular the products $\lambda\mu$ and $\alpha\beta$ and their trace functions,
$t_{\lambda\mu}$ and $t_{\alpha\beta}$,
that depend on the orientations. 

\begin{Lemma}
\label{lemma:difindep}
Up to sign, the difference  $t_{\lambda\mu}-t_{\alpha\beta}$
 is independent of the choice of orientations of $\lambda$ and $\mu$. The sign depends on
whether we change one (-) or both (+) orientations.
\end{Lemma}

\begin{Proposition}
\label{Prop:symplectic40} 
Let $\rho_0 \in R^*(S, \SL2)$ be a good, $\partial$-regular representation
such that $t_{\lambda\mu}(\rho_0)\neq t_{\alpha\beta}(\rho_0)$. Then:
\begin{enumerate}[(1)]
 \item the map $T=(t_\lambda,t_\mu)\colon  \R^*(S_{0,4},\partial S_{0,4}, \SL2)_{\rho_0}\setminus\{t_{\lambda\mu}=t_{\alpha\beta}\} \to \CC^2 $ is  a local parameterization.
 \item 
 the symplectic form on $\R^*(S_{0,4},\partial S_{0,4}, \SL2)_{\rho_0}\setminus\{t_{\lambda\mu}=t_{\alpha\beta}\}$ is the pullback $T^*\omega$ where
\[
  \omega = \pm   \frac{d\, t_\lambda\wedge d\, t_\mu}{  t_{\lambda\mu}-t_{\alpha\beta} } .
 \]
\end{enumerate} 
\end{Proposition}

We fix the notation for both proofs. 
The fundamental group of $S_{0,4}$ is freely generated by three elements $\gamma_1$, $\gamma_2$, and $\gamma_3$,
and the peripheral curves are represented by  $\gamma_1$, $\gamma_2$, $\gamma_3$, and  $\gamma_1\gamma_2\gamma_3$,
see Figure~\ref{Fig:loops}.
We shall assume that the orientations are so that $\lambda=\gamma_1\gamma_2$ and $\mu=\gamma_2\gamma_3$. 
With this choice of orientation, $\alpha= \gamma_1\gamma_2\gamma_3$ and $\beta=\gamma_2$, so
$$
t_{\lambda\mu}-t_{\alpha\beta}=  t_{1223}-t_{1232}.
$$

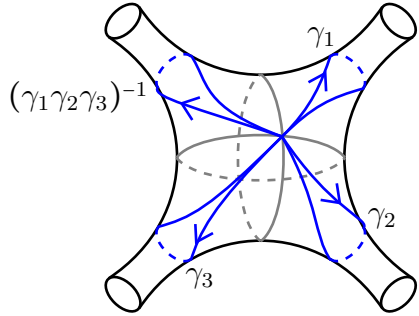
\begin{figure}[h]
\begin{center}
\begin{tikzpicture}[scale=.4]


\draw[line width=1pt] (-4, 5) .. controls (-2,2) and (2,2) ..  (4, 5);
\draw[line width=1pt] (5, -4) .. controls (2,-2) and (2,2) ..  (5,4);
\draw[line width=1pt] (-4, -5) .. controls (-2,-2) and (2,-2) ..  (4, -5);
\draw[line width=1pt] (-5, -4) .. controls (-2,-2) and (-2,2) ..  (-5,4);
\draw[line width=1pt,rotate=45]  (6,0) arc (-180:180:.4 and .7);
\draw[line width=1pt,rotate=135]  (6,0) arc (-180:180:.4 and .7);
\draw[line width=1pt,rotate=225]  (6,0) arc (-180:180:.4 and .7);
\draw[line width=1pt,rotate=-45]  (6,0) arc (-180:180:.4 and .7);
   \draw[line width=1pt, gray, dashed] (0, -2.77) .. controls (-1,-2) and  (-1,2) ..  (0, 2.77);
  \draw[line width=1pt, gray, dashed] ( -2.77, 0) .. controls (-2,-1) and (2,-1) ..  (2.77,0);
 \draw[line width=1pt, gray] (0, -2.77) .. controls (1,-2) and  (1,2) ..  (0, 2.77);
  \draw[line width=1pt, gray] ( -2.77, 0) .. controls (-2,1) and (2,1) ..  (2.77,0);
\draw[line width=1pt, blue] (0.7,0.7) .. controls (1.7,1.7) and (2,2.5) .. (2.3,3.4);
\draw[line width=1pt, blue] (0.7,0.7) .. controls (1.7,1.7) and (2.3,2) .. (3.4,2.3);
\draw[line width=1pt,blue, dashed, rotate=45]  (4,-.75) arc (-90:90:.4 and .75);
\draw[line width=1pt, blue] (1.57,2.5)--(2.12,2.8)--(2.15,2.2);
\node at (2,4) {$\gamma_1$};
\draw[line width=1pt, blue] (0.7,0.7) .. controls (2,-1.7) and (1.7,-2.5) .. (2.3,-3.4);
\draw[line width=1pt, blue] (0.7,0.7) .. controls (2.3,-1.7) and (3.5,-2) .. (3.4,-2.3);
\draw[line width=1pt,blue, dashed, rotate=-45]  (4,-.75) arc (-90:90:.4 and .75);
\draw[line width=1pt, blue] (1.97,-1.4)--(2.60,-1.5)--(2.50,-.85);
\node at (4,-2) {$\gamma_2$};
\draw[line width=1pt, blue] (0.7,0.7) .. controls (-1.7,-1.7) and (-2,-2.5) .. (-2.3,-3.4);
\draw[line width=1pt, blue] (0.7,0.7) .. controls (-1.7,-1.7) and (-2.3,-2) .. (-3.4,-2.3);
\draw[line width=1pt,blue, dashed, rotate=-135]  (4,-.75) arc (-90:90:.4 and .75);
\draw[line width=1pt, blue] (-1.57,-2.5)--(-2.12,-2.8)--(-2.15,-2.2);     
\node at (-2,-4) {$\gamma_3$};
\draw[line width=1pt, blue] (0.7,0.7) .. controls (-2,1.7) and (-1.7,2.5) .. (-2.3,3.4);
\draw[line width=1pt, blue] (0.7,0.7) .. controls (-2.3,1.7) and (-3.5,2) .. (-3.4,2.3);
\draw[line width=1pt,blue, dashed, rotate=135]  (4,-.75) arc (-90:90:.4 and .75);
\draw[line width=1pt, blue] (-2,2.0)--(-2.60,1.8)--(-2.25,1.25);
\node at (-6,2) {$(\gamma_1\gamma_2\gamma_3)^{-1}$};
%
%
%
%
%
%
\end{tikzpicture}
\end{center}
\caption{The loops $\gamma_1$, $\gamma_2$, and $\gamma_3$ that represent the generators of  $\pi_1(S_{0,4})$.}
\label{Fig:loops}
\end{figure}

\begin{proof}[Proof of Lemma~\ref{lemma:difindep}]
It suffices to change the orientation of a single curve, so we follow the examples of Figures~\ref{Fig:ab}
and \ref{Fig:ba}.
 If we change the orientation of $\mu$ then  $\mu=\gamma_3^{-1}\gamma_2^{-1}$,  $\alpha= \gamma_1$,
 and $\beta= \gamma_3^{-1}$.
 We aim to prove
 \begin{equation}
 \label{eqn:t1223-t1232}
 t_{1223}-t_{1232}= - (t_{12\bar3\bar2}-t_{1\bar 3})
\end{equation}
(with negative sign, because we change the orientation of a single curve). 
From Equality \eqref{eqn:trace}  
we have:
\begin{equation}
 \label{eqn:t12t23}
\begin{array}{rcl}
 t_{12}t_{23} & = & t_{1223}+t_{12\bar3\bar2} \\
 t_{12}t_{32} & = & t_{1232}+ t_{1\bar 3},
\end{array} 
\end{equation}
Then equality \eqref{eqn:t1223-t1232} follows by subtracting in \eqref{eqn:t12t23} and using $t_{32}=t_{23}$.
 \end{proof}

 \begin{proof}[Proof of Proposition~\ref{Prop:symplectic40}]

(1) We shall use a computation in cohomology, first by cutting the surface $S_{0,4}$ along $\lambda=\gamma_1\gamma_2$ into two pairs of pants
$P_1$ and $P_2$, with $\pi_1(P_1)=\langle \gamma_1,\gamma_2 \rangle $ and 
$\pi_1(P_2)=\langle \gamma_3,\gamma_1\gamma_2 \rangle $. 
Notice that for $[\rho]\in \R^*(S_{0,4},\partial S_{0,4}, \SL2)_{\rho_0}\setminus\{t_{\lambda\mu}=t_{\alpha\beta}\}$ we have that $\rho\vert_{ \pi_1(P_i)}$ is nonabelian.
Suppose that, contrary to our claim,  $\rho\vert_{ \pi_1(P_1)}$ is abelian that is
 $\rho(\gamma_1)$ and $\rho(\gamma_2)$ commute. Then 
\[
t_{\lambda\mu}([\rho]) = t_{1223}([\rho]) = t_{2123}([\rho])=t_{1232}([\rho]) = t_{\alpha\beta}([\rho])\,,
\]
contradicting
the hypothesis. This argument also shows that $\rho\vert_{ \pi_1(P_2)}$ is nonabelian.

As $\rho\vert_{ \pi_1(P_i)}$ is nonabelian, $H^0(\pi_1(P_i);\operatorname{Ad} {\rho})=0$. Hence, we obtain the following Mayer-Vietoris exact sequence:
$$
0\to
H^0(\lambda;\operatorname{Ad} {\rho})\overset \beta \to H^1(S;\operatorname{Ad} {\rho})\to 
H^1(P_1 ;\operatorname{Ad} {\rho})\oplus
H^1(P_2; \operatorname{Ad} {\rho})\to H^1(\lambda ;\operatorname{Ad} {\rho})\to 0\,.
$$
Using the local parameterization of a pair of paints, this sequence yields that the tangent space to $\R^*(S,\SL2)$ at $\rho$
is generated by the infinitesimal deformations $\partial_{t_1}$, $\partial_{t_2}$, $\partial_{t_3}$, $\partial_{t_{123}}$,
$\partial_{t_{12}}$ and $\beta(a)$, 
where $0\neq a\in H^0(\lambda;\operatorname{Ad} {\rho})\cong \sl2^{\operatorname{Ad}\rho(\gamma) }$.
Hence, the tangent space to  $\R^*(S,\partial S,\SL2)_{\rho_0}$ at $\rho$ is generated by $\partial_{t_{12}}=\partial_{t_\lambda}$
and $\beta(a)$. Notice that $d\, t_{\lambda} ( \beta(a) )=0$ since
$\beta(a)$ is an infinitesimal bending along $\lambda$. In order to prove that $(t_\mu,t_\lambda)$ are local parameters at $\rho$
we must show that
$d\, t_{\mu} ( \beta(a) )\neq 0 $.

Next we compute $d\, t_{\mu} ( \beta(a) )$. By setting $A_i=\rho(\gamma_i)$, we obtain
$\rho(\lambda)=A_1A_2$ and
we can chose $a=\frac{1}{2}(A_1A_2-A_2^{-1}A_1^{-1} )$.
As $\lambda$ is a separating curve,  the infinitesimal bending is the derivative  respect to $\varepsilon$ of the path of representations:
$$
\begin{array}{rcl}
 \gamma_1& \mapsto &  A_1 \\
 \gamma_2& \mapsto & A_2 \\
 \gamma_3& \mapsto & \big(1+\varepsilon a+ o(\varepsilon)\big)  A_3 \big(1-\varepsilon a+o(\varepsilon)\big)  , 
\end{array}
$$
see \cite[Lemma~5.1]{JohnsonMillson} for details. 
Since $\mu=\gamma_2\gamma_3$ is mapped to $A_2A_3+\varepsilon (A_2aA_3 -A_2A_3 a)+
o(\varepsilon)$, 
we have
$$
d\operatorname{t}_\mu(\beta(a)) =\operatorname{tr}( A_2aA_3-A_2A_3a)=
\frac{1}{ 2}  ( t_{2123}-t_{\bar 13}-t_{2312}+t_{23\bar 2\bar 1}  ) .
$$
By Lemma~\ref{lemma:difindep} and its proof, and using  that the trace is invariant by cyclic permutations and by taking the inverse:
\begin{align*}
t_{\alpha\beta}-t_{\lambda\mu}&=t_{1232}-t_{1223}=t_{2123}-t_{2312}  \\ &=t_{12\bar3\bar2}-t_{1\bar 3}=t_{23\bar2\bar1}  -t_{\bar 1 3} . 
\end{align*}
Thus
$ d\operatorname{t}_\mu(\beta(a)) =t_{\alpha\beta}-t_{\lambda\mu}\neq 0 $. This proves Assertion~(1) of the proposition.
\smallskip

\begin{figure}[h]
\begin{center}
\begin{tikzpicture}[scale=.4]

\begin{scope}[shift={(-14,0)}]
\draw[line width=1pt] (-4, 5) .. controls (-2,2) and (2,2) ..  (4, 5);
\draw[line width=1pt] (5, -4) .. controls (2,-2) and (2,2) ..  (5,4);
\draw[line width=1pt] (-4, -5) .. controls (-2,-2) and (2,-2) ..  (4, -5);
\draw[line width=1pt] (-5, -4) .. controls (-2,-2) and (-2,2) ..  (-5,4);
\draw[line width=1pt,rotate=45]  (6,0) arc (-180:180:.4 and .7);
\draw[line width=1pt,rotate=135]  (6,0) arc (-180:180:.4 and .7);
\draw[line width=1pt,rotate=225]  (6,0) arc (-180:180:.4 and .7);
\draw[line width=1pt,rotate=-45]  (6,0) arc (-180:180:.4 and .7);

   \draw[line width=1pt, red, dashed] (0, -2.77) .. controls (-1,-2) and  (-1,2) ..  (0, 2.77);
  \draw[line width=1pt, blue, dashed] ( -2.77, 0) .. controls (-2,-1) and (2,-1) ..  (2.77,0);
 \draw[line width=1pt, red] (0, -2.77) .. controls (1,-2) and  (1,2) ..  (0, 2.77);
  \draw[line width=1pt, blue] ( -2.77, 0) .. controls (-2,1) and (2,1) ..  (2.77,0);
 \draw[line width=1pt, red] (0.4,1.2)--(0.54,1.8)--(0.95,1.25);
  \draw[line width=1pt, blue] (1.2,0.4)--(1.8,0.54)--(1.25,0.95);
\node at (3.5,0) {$\mu$};
\node at (0,3.5) {$\lambda$};
\end{scope}

\draw[line width=1pt] (-4, 5) .. controls (-2,2) and (2,2) ..  (4, 5);
\draw[line width=1pt] (5, -4) .. controls (2,-2) and (2,2) ..  (5,4);
\draw[line width=1pt] (-4, -5) .. controls (-2,-2) and (2,-2) ..  (4, -5);
\draw[line width=1pt] (-5, -4) .. controls (-2,-2) and (-2,2) ..  (-5,4);
\draw[line width=1pt,rotate=45]  (6,0) arc (-180:180:.4 and .7);
\draw[line width=1pt,rotate=135]  (6,0) arc (-180:180:.4 and .7);
\draw[line width=1pt,rotate=225]  (6,0) arc (-180:180:.4 and .7);
\draw[line width=1pt,rotate=-45]  (6,0) arc (-180:180:.4 and .7);


   \draw[line width=1pt, red, dashed] (-.7, -.7) .. controls   (-.6,2) ..  (0, 2.77);
 \draw[line width=1pt, red] (.7, .7) .. controls   (.5,2) and (.3,2.5) ..  (0, 2.77);
  \filldraw (.7,.7) circle (3pt);
    \filldraw (-.7,-.7) circle (3pt);
   \node at (1.4,.7) {$p_1$}; 
      \node at (-1.5,-.7) {$p_2$}; 
\end{tikzpicture}
\end{center}
\caption{The intersection points $p_1,p_2$ and the arc between them.}
\label{Fig:intersection}
\end{figure}
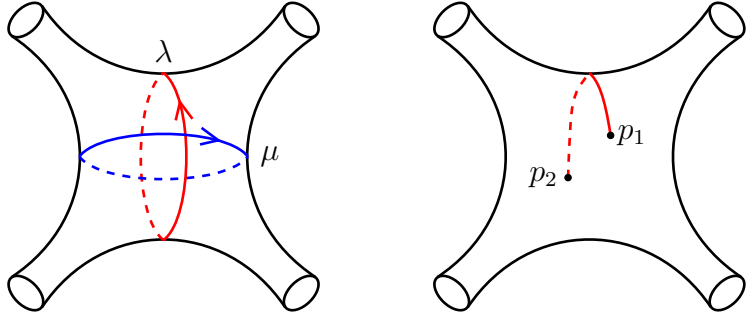

We next prove Assertion~(2).
Let $[\rho]\in \R^*(S_{0,4},\partial S_{0,4}, \SL2)_{\rho_0}\setminus\{t_{\lambda\mu}=t_{\alpha\beta}\}$, and set again $A_i=\rho(\gamma_i)$.
We apply Proposition~\ref{Prop:GoldmanPoisson} to compute the Poisson bracket $\{t_\lambda,t_\mu\}([\rho])$.
 
The curves $\lambda$ and $\mu$ intersect in two points,
$p_1$ and $p_2$, in Figure~\ref{Fig:intersection}. Let $p_1$ be the base point of the fundamental group used in Figure~\ref{Fig:loops}. 
The contribution of $p_1$ is
$$
\epsilon \operatorname{tr}(\operatorname{\mathbf T}(A_1A_2) \operatorname{\mathbf T}(A_2A_3))=
\frac{\epsilon }2\big(t_{1223}([\rho]) -t_{12\bar3\bar 2}([\rho])\big) 
$$
for some  $\epsilon=\pm 1$.
To compute the contribution of $p_2$ we consider an arc from $p_1$ to $p_2$ to relate the base points between fundamental groups.
Assume that this arc is half of $\lambda$, as in Figure~\ref{Fig:intersection}, then 
$\rho_2(\lambda)=\rho(\lambda)=A_1A_2$
and $\rho_2(\mu)=\rho( \gamma_1\gamma_2\gamma_3 \gamma_1^{-1})= A_1 A_2 A_3 A_1^{-1}$. In addition, the orientation of the 
intersection is opposite to the previous one, hence the contribution of $p_2$ is
$$
-\epsilon 
\operatorname{tr}(\operatorname{\mathbf T}(A_1A_2) \operatorname{\mathbf T}( A_1 A_2 A_3  A_1^{-1}))=- \frac{\epsilon }2 \big( t_{2123}([\rho]) -t_{1\bar 3}([\rho])\big)
$$
Hence, for $\mathcal{B}(X,Y)=-\operatorname{tr}(X Y)$ we obtain from Proposition~\ref{Prop:GoldmanPoisson}
\[
\{t_\lambda,t_\mu\}=
-\frac\epsilon2\,(t_{1223} -t_{12\bar3\bar 2}-t_{2123} +t_{1\bar 3}) \,,
\]
and by equation \eqref{eqn:t1223-t1232} we have 
\[
t_{\lambda\mu} - t_{\alpha\beta} = t_{1223} -t_{2123}= -t_{12\bar3\bar 2} +t_{1\bar 3}\,.
\]

Finally, the formula for the symplectic form on the coordinates $(t_\lambda,t_\mu)$ 
follows again from equation~\eqref{eqn:poisson}.
\end{proof}

\section{Volume forms for free groups of higher rank}
\label{section:higherrank}

\subsection{Volume form on $\R^*(F_k,\SL2) $}

We recall the notation $t_{i_1\cdots i_k}$
for the trace function $\operatorname{tr}_\gamma$ of $\gamma=\gamma_{i_1}\cdots \gamma_{i_k}$,
with the convention 
$\gamma_{\bar i}=\gamma_i ^{-1}$.

We start discussing the volume form for the free group of rank three.
Following \cite{AcunaMontesinos},
the variety of characters $X(F_3,\SL2)$ is a  branched covering of $\CC^{6}$. 
More precisely, the branched covering  is given by trace functions: 
\begin{equation}
\label{eqncoordinatesFn}
T = (t_1, t_2, t_{3},   t_{12},  t_{13}, t_{23})\colon X(F_3,\SL2)\to \CC^{6}\, .
 \end{equation}
The branching is given by the variables $t_{123}$ and $t_{213}$, 
as they are the solutions of the quadratic equation
$$
z^2-R z+S=0
$$
for 
\begin{align}
\label{eqn:P}
 R &= t_1 t_{23} + t_2 t_{13} + t_3 t_{12}- t_1 t_2 t_3  \\
\label{eqn:Q}
 S &= t_1^2+t_2^2+t_3^2+t_{12}^2+t_{13}^2+t_{23}^2+t_{12}t_{13}t_{23}- t_1t_2t_{12}- t_1t_3t_{13}- t_2t_3t_{23}-4  .
\end{align}
Recall that the trace is invariant by cyclic permutation of the group elements:
\[
 t_{123}=t_{231}=t_{312} .
 \]
The branching locus is defined by $t_{123}= t_{213}$. Away from it, the variables \eqref{eqncoordinatesFn}
define local coordinates.

\begin{Proposition} 
\label{Prop:volFn} For
\[
T = (t_1, t_2, t_{3}, t_{12},  t_{13}, t_{23})\colon \R^*(F_3,\SL2)\to \CC^{6}
\]
the restriction of the volume form 
to the open subset
$\R^*(F_3,\SL2)\setminus\{ t_{123}= t_{213}\}$ is the pull-back form $ \Omega^{\SL2}_{F_3}=\pm T^*\Omega$, where
 \begin{equation}
 \label{eqn:OmegaF3}
 \Omega = \pm \frac{4}{t_{123}-t_{213} }  d\,t_1 \wedge d\,t_2 \wedge d\,t_3 \wedge d\,t_{12}  
 \wedge  d\,t_{13} \wedge d\,t_{23} \, .
 \end{equation}
\end{Proposition}

\begin{proof}
Consider the surface $S=S_{0,4}$.
 Since $\gamma_1$, $\gamma_2$, $\gamma_3$, and $\gamma_1\gamma_2\gamma_3$ are the peripheral elements, using Proposition~\ref{Prop:symplectic40}
 and Corollary~\ref{cor:dtrace}, 
  \begin{equation}
 \label{eqn:OmegaF3prel}
  \Omega_{F_3}^{\SLL_2}=  \Omega_{S_{0,4}}^{\SLL_2} = \pm    4 \frac{ d\,t_{12}  \wedge d\,t_{23}} { t_{1223}-t_{1232} }  \wedge   d\,t_1 \wedge d\,t_2 \wedge d\,t_3 \wedge d\,t_{123}.
 \end{equation}
It remains to replace $d\,t_{123}$ by $d\,t_{13}$ in this formula.
Differentiating the equality
$$
t_{123}^2-R\, t_{123}+S =0,
$$
where $R$ and $S $ are given in \eqref{eqn:P} and \eqref{eqn:Q}, we deduce:
 \begin{equation}
 \label{eqn:dt123}
\left(2 t_{123}- R \right ) d\, t_{123} =
\sum_{\eta\in \{1,2,3,12,13,23\} }\left( \frac{\partial R}{\partial t_\eta}t_{123} - \frac{\partial S }{\partial t_\eta}\right) d\, t_\eta\, .
 \end{equation}
Since $R= t_{123}+t_{213}$,
  \begin{equation}
 \label{eqn:2t123}
2 t_{123}- R= t_{123}-t_{213}
  \end{equation}
In addition, using $t_1 t_3= t_{13}+ t_{1\bar 3}$,
\begin{equation}
 \label{eqn:dp3t-q}
\frac{\partial R}{\partial t_{13}}t_{123} - \frac{\partial S }{\partial t_{13}}= 
t_2t_{123}-(2 t_{13} + t_{12} t_{23}-t_1 t_3)=  t_2t_{123}-  t_{12} t_{23}-t_{13} +t_{1\bar 3} 
\end{equation}
Using the standard relations on traces, we have:
\begin{align}
 t_{13} & = t_{12\bar23}= t_{312\bar 2}= t_{312}t_{\bar 2}-t_{3122}= t_{123} t_2-t_{1223} \\
 t_{1\bar 3} & = t_{12\bar2\bar 3}=  t_{12}t_{3 2}-t_{1232}= t_{12}t_{ 23 }-t_{1232}
\end{align}
From those equations:
\begin{equation}
 \label{eqn:dp3t-qshort}
\frac{\partial R}{\partial t_{13}}t_{123} - \frac{\partial S }{\partial t_{13}}= 
t_{1223}-t_{1232}\, .
\end{equation}
Hence, using \eqref{eqn:2t123} and \eqref{eqn:dp3t-qshort},  equality \eqref{eqn:dt123} becomes:
\begin{equation}
\label{eqn:difsts}
( t_{123}-t_{213} ) d\, t_{123} = (t_{1223}-t_{1232}) d\, t_{13} + \sum_{\eta\in \{1,2,3,12,23\} } p_\eta\, d\, t_\eta\
\end{equation}
for some polynomials $p_\eta$.
Using \eqref{eqn:difsts} to replace $d\, t_{123}$ by $d\, t_{13}$ in \eqref{eqn:OmegaF3prel}, we prove \eqref{eqn:OmegaF3}. 
\end{proof}

\begin{proof}[Proof of Theorem~\ref{thm:F3SL2}]
Write $F_k=\langle\gamma_1,\gamma_2,\ldots,\gamma_k\rangle$ and
consider the graph $\mathcal G$ with one vertex and $k$ edges, so that $\pi_1(\mathcal G)\cong F_k$. Consider subgraphs $\mathcal G'$ and $\mathcal G''$, so that
$\pi_1(\mathcal G')=\langle\gamma_1,\gamma_2,\ldots,\gamma_{k-1}\rangle$ and   $\pi_1(\mathcal G'')=\langle\gamma_1,\gamma_2,\gamma_{k}\rangle$;
therefore $\mathcal G =\mathcal G'\cup \mathcal G''$ and $\pi_1( \mathcal G'\cap  \mathcal G'')=\langle\gamma_1,\gamma_2\rangle$.
Since we assume  $t_{12\bar 1\bar 2}\neq 2$, $\rho(\pi_1  ( \mathcal G'\cap  \mathcal G''))$ is irreducible, therefore, the long exact sequence of
Mayer-Vietoris applied to $(\mathcal G', \mathcal G'') $ is:
\begin{equation}
0 \to H^1(\mathcal G,\operatorname{Ad}\rho)\to  H^1(\mathcal G',\operatorname{Ad}\rho)\oplus  H^1(\mathcal G'',\operatorname{Ad}\rho)\to  
 H^1(\mathcal G'\cap \mathcal G'',\operatorname{Ad}\rho)\to 0 .
\end{equation}
Interpreting cohomology groups as tangent spaces to spaces of representations, the assertion on the local parameterization is straightforward
from the sequence.
By an induction argument, the formula for the volume form is a consequence of the product of torsions, Corollary~\ref{cor:C3} and
Proposition~\ref{Prop:volFn}.
\end{proof}

\subsection{Volume form on $\R^*(F_k,\SLL_3(\CC)) $}

Before proving Proposition~\ref{prop:FkSL3} and Theorem~\ref{thm:FkSL3}, we need two lemmas on regular elements in $\SLL_3(\CC)$. 
Recall that an element of $\SLL_3(\CC)$ is regular if its minimal polynomial and its characteristic polynomial have the same degree. 
This is the case if and only if each eigenspace is one-dimensional.

\begin{Lemma}
\label{Lemma:regtr}
Let $A,B\in \SLL_3(\CC)$. If $\operatorname{tr}(ABA^{-1}B^{-1})\neq \operatorname{tr}(BAB^{-1}A^{-1})$
then:
\begin{enumerate}[(i)]
 \item both $A$ and $B$ are regular and
 \item the subgroup 
$\langle A,B \rangle \subset \SLL_3(\CC)$ is irreducible.
\end{enumerate}

%
\end{Lemma}

\begin{proof}
(i) Assume that $A$ is not regular. Then it  has an eigenvalue $\lambda\in\CC^*$ 
with an eigenspace $E_\lambda=\ker(A-\lambda\operatorname{Id})$ of dimension  $\dim E_\lambda\geq 2$. Therefore
$\dim  (E_\lambda\cap B(E_\lambda))\geq 1$. Chose a 
nonzero vector $v\in E_\lambda\cap B(E_\lambda)$, by construction $B^{-1}(v)\in E_\lambda$ and
$(ABA^{-1}B^{-1})(v)= v$. This yields that $1$ is an eigenvalue of the commutator  $ABA^{-1}B^{-1}$, therefore it has the same eigenvalues as its inverse,
which implies that $\operatorname{tr}(ABA^{-1}B^{-1})=\operatorname{tr}(BAB^{-1}A^{-1})$.

(ii) By contradiction, assume that $L\subset\CC^3$ is a proper subspace invariant by both $A$ and $B$. If $\dim L=1$, then this is and eigenspace of
$ABA^{-1}B^{-1}$ with eigenvalue $1$, and if $\dim L=2$, by looking at the action on $\CC^3/L$ we also deduce that $1$ is an
eigenvalue of $ABA^{-1}B^{-1}$. Therefore, by the discussion on the previous item, this contradicts the hypothesis  
$\operatorname{tr}(ABA^{-1}B^{-1})\neq \operatorname{tr}(BAB^{-1}A^{-1})$.
\end{proof}

\begin{Lemma}
\label{Lemma:reginv}
 Let $A\in \SLL_3(\CC)$. If $A$ is regular then the $\Ad_A$-invariant subspace of $\sll_3(\CC)$ is
 $$\sll_3(\CC)^{\Ad_A}=\langle A-\tfrac{\operatorname{tr}(A)}3\operatorname{Id},
 A^{-1}-\tfrac{\operatorname{tr}(A^{-1})}3\operatorname{Id}\rangle .$$  
\end{Lemma}

\begin{proof}
It is clear from construction that both $A-\frac{\operatorname{tr}(A)}3\operatorname{Id}$ and
 $A^{-1}-\frac{\operatorname{tr}(A^{-1})}3\operatorname{Id}$ are $\Ad_A$-invariant. 
 All we need to show is that those elements are linearly independent, as  by regularity $\dim \sll_3(\CC)^{\Ad_A}=2$.
 If  $ A-\frac{\operatorname{tr}(A)}3\operatorname{Id}$ and $ A^{-1}-\frac{\operatorname{tr}(A^{-1})}3\operatorname{Id}$ were  linearly dependent,
then   $A$, $\operatorname{Id}$, and $A^{-1}$ would satisfy a nontrivial linear relation.
Multiplying it by $A$, 
the same relation would be satisfied by $A^2$, $A$ and $\operatorname{Id}$,
and hence $A$ would have an eigenspace of dimension at least $2$,
 contradicting regularity. 
\end{proof}

%
%
%

\begin{Remark} \label{Remark:good}
It follows from Schur's Lemma \cite{FultonHarris} that every irreducible representation
$\rho\colon\Gamma\to \SLL_N(\CC)$ is good, that is the centralizer of $\rho(\Gamma)$ coincides with the center of $\SLL_N(\CC)$.
\end{Remark}

\begin{proof}[Proof of Proposition~\ref{prop:FkSL3}]
Assume $k=3$, the general case follows from an induction argument as in the proof of Theorem~\ref{thm:F3SL2}.

We chose generators $F_3 =\langle \gamma_1,\gamma_2,\gamma_3 |-\rangle$ and we identify $F_3$ with
$\pi_1(S_{0,4})$. We represent $S_{0,4}$  as the union of two pairs of pants $P'$ and $P''$, so that $P'\cap P''$ is a circle.
Chose the generators of the fundamental group so that 
$\pi_1 (P')=\langle \gamma_1,\gamma_2 \rangle$, $\pi_1 (P'')=\langle \gamma_1,\gamma_3 \rangle$,
and $\gamma_1$ is the generator of $\pi_1(   P'\cap P'' ) $. Then the peripheral elements of $P'$ are
$\gamma_1$, $\gamma_2$, and $\gamma_1\gamma_2$, and those of $P''$,  $\gamma_1$, $\gamma_3$, and $\gamma_1\gamma_3$. The peripheral elements of $S$ are $\gamma_2$, $\gamma_3$, $\gamma_1\gamma_2$, and $\gamma_1\gamma_3$.

Let $[\rho]\in \R^*(F_3,\SLL_3(\CC))\setminus    \{ \uptau_{12\bar 1\bar{2}}= \uptau_{21\bar 2\bar 1} \}\cup
\{ \uptau_{13\bar 1\bar{3}}= \uptau_{31\bar 3\bar 1} \} \cup\{
  \varDelta^1_{23} = 0\} $ be a representation where
\begin{equation}
\label{eqn:vardelta}
\varDelta^1_{23} = (\uptau_{123}-\uptau_{132}) (\uptau_{\bar 1\bar 2\bar 3}-\uptau_{\bar 1\bar 3\bar 2}) -
  (\uptau_{1\bar 2\bar 3}-\uptau_{1\bar 3\bar 2}) (\uptau_{\bar 1 2  3}-\uptau_{\bar 1 3 2 }). 
\end{equation}
We have to show that the 16 functions
\[
 (\uptau_1, \uptau_{\bar 1},\uptau_2, \uptau_{\bar 2},\uptau_3, \uptau_{\bar 3}, \uptau_{12}, \uptau_{\bar 1\bar 2}, \uptau_{13}, \uptau_{\bar 1\bar 3},
     \uptau_{23}, \uptau_{\bar 2\bar 3},
  \uptau_{1\bar 2}, \uptau_{\bar 1 2},  \uptau_{1\bar 3}, \uptau_{\bar 1 3}) 
\]
define a local parameterization at $[\rho]$.
The hypothesis
$\operatorname{tr}(\rho([\gamma_1,\gamma_i])])\neq \operatorname{tr}(\rho([\gamma_i,\gamma_1]))$ for $i=2,3$  implies that $\rho(\gamma_j)$, $j=1,2,3$, are regular elements
(Lemma~\ref{Lemma:regtr}). It follows also that $\rho(\gamma_{1}\gamma_2)$
and $\rho(\gamma_{1}\gamma_3)$ are regular since
$\operatorname{tr}(\rho([\gamma_1\gamma_i,\gamma_1])) = \operatorname{tr}(\rho([\gamma_i,\gamma_1]))$
and
$\operatorname{tr}(\rho([\gamma_1,\gamma_1\gamma_i])) = \operatorname{tr}(\rho([\gamma_1,\gamma_i]))$ for $i=2,3$.
The Mayer-Vietoris long exact sequence is:
\begin{multline}
\label{eqn:lesp'p''}
0\to H^0(\gamma_1,\Ad\rho)
\xrightarrow{\beta} H^1(S,\Ad\rho)
\xrightarrow{j} H^1(P',\Ad\rho)\oplus H^1(P'',\Ad\rho)\\
\xrightarrow{\Delta} H^1(\gamma_1,\Ad\rho)\to 0\,.
\end{multline}
Chose  $\mathbf{u}$  a basis  for $H^0(\gamma_1,\Ad\rho) $.
We will proceed as in the proof of Proposition~\ref{Prop:symplectic40}.
Viewing the cohomology groups as tangent spaces, the proposition  will follow from the local parameterizations for the representation space of $P'$ and $P''$, and  
from \eqref{eqn:lesp'p''}, provided we show that 
\begin{equation*}
d\, \uptau_{23}\wedge d\, \uptau_ {\bar 2\bar 3} (\wedge \beta(\mathbf{u}))\neq 0 .
\end{equation*}
We prove below in Lemma~\ref{lemma:nuDelta} that  
$ d\, \uptau_{23}\wedge d\, \uptau_ {\bar 2\bar 3} (\wedge \beta(\mathbf{u}))=\pm \varDelta^1_{23}$, which is nonzero by hypothesis.
\end{proof}

\begin{Lemma}\label{lemma:nuDelta}
$  d\, \uptau_{23}\wedge d\, \uptau_ {\bar 2\bar 3} (\beta(\wedge\mathbf{u}))=\pm \varDelta^1_{23}$, where $\varDelta^1_{23}$ is as in \eqref{eqn:vardelta}.
 \end{Lemma}

\begin{proof}
Set $A_1 = \rho(\gamma_1)$. By Lemma~\ref{Lemma:reginv}, the elements
\[
x=A_1-\frac{\operatorname{tr}A_1 }{3}\operatorname{Id}\quad \textrm{ and }\quad  
y=A_1^{-1}-\frac{\operatorname{tr}A_1^{-1}}{3}\operatorname{Id}
\]
form a basis of the invariant subspace $\sll_3(\CC)^{\operatorname{Ad}\rho(\gamma_1)}$.
We choose $\mathbf{u}=\{x,y\}$
via the isomorphism
$H^0(\gamma_1,\operatorname{Ad}\rho)\cong \sll_3(\CC)^{\operatorname{Ad}_{A_1}}$.

Then $\beta(x)$ is the tangent vector to the infinitesimal bending:
$$
\gamma_1\mapsto A, \quad \gamma_2\mapsto B, \quad \gamma_3\mapsto ( \operatorname{Id}+\varepsilon x )C(\operatorname{Id}-\varepsilon x )= 
C+\varepsilon(xC-Cx)
\qquad \textrm{ in }\CC[\varepsilon]/\varepsilon^2
$$
and similarly for $\beta(y)$. To compute $d\,\uptau_{23}$ and $d\,\uptau_{\bar 2\bar 3}$ on  $\beta(x)$
and $\beta(y)$, we must evaluate the infinitesimal deformations
on $\gamma_2\gamma_3$ and $\bar\gamma_2\bar\gamma_3$. Thus the path corresponding to  $\beta(x)$ evaluated at $\gamma_2\gamma_3$ is 
\begin{equation}
 \label{eqn:g2g3}
 \gamma_2\gamma_3\mapsto BC+\varepsilon(BxC-BCx)= BC+\varepsilon(BAC-BCA) .
\end{equation}
Therefore, taking traces we get:
\begin{equation}
 \label{eqn:dt23x}
 d\,\uptau_{23}(\beta(x))=\uptau_{213}-\uptau_{231}=\uptau_{132}-\uptau_{123} .
\end{equation}
The same argument for $y$ instead of $x$ gives:
\begin{equation}
 \label{eqn:dt23y}
d\,\uptau_{23}(\beta(y))=\uptau_{2\bar 13}-\uptau_{23\bar 1}=\uptau_{\bar 132}-\uptau_{\bar 123} .
\end{equation}
To evaluate $d\,\uptau_{\bar 2\bar 3}=d\,\uptau_{\bar 3\bar 2}$, we take inverses in \eqref{eqn:g2g3}
\begin{equation}
 \label{eqn:barg3g3}
 (\gamma_2\gamma_3)^{-1}\mapsto C^{-1}B^{-1}+\varepsilon(xC^{-1}B^{-1}-C^{-1}xB^{-1})= C^{-1}B^{-1}+\varepsilon(AC^{-1}B^{-1}-C^{-1}AB^{-1})
\end{equation}
and taking traces we get:
\begin{equation}
 \label{eqn:dtbar23x}
d\,\uptau_{\bar 2\bar 3}(\beta(x))=\uptau_{1\bar 3\bar 2}-\uptau_{\bar 31\bar 2}=\uptau_{1\bar 3\bar 2}-\uptau_{1\bar 2\bar 3} .
\end{equation}
Again the same argument for $y$ instead of $x$ gives:
\begin{equation}
 \label{eqn:dtbar23y}
 d\,\uptau_{\bar 2\bar 3}(\beta(y))=\uptau_{\bar 1\bar 3\bar 2}-\uptau_{\bar 3\bar 1\bar 2}=\uptau_{\bar 1\bar 3\bar 2}-\uptau_{\bar 1\bar 2\bar 3} .
\end{equation}
Hence 
\begin{align*}
d\,\uptau_{23}\wedge d\,\uptau_{\bar 2\bar 3}(\beta(x)\wedge\beta(y) )
&=\pm  
(\uptau_{123}-\uptau_{132}) (\uptau_{\bar 1\bar 2\bar 3}-\uptau_{\bar 1\bar 3\bar 2}) - 
(\uptau_{1\bar 2\bar 3}-\uptau_{1\bar 3\bar 2}) (\uptau_{\bar 1 2  3}-\uptau_{\bar 1 3 2 }) \\
&=\pm
\varDelta^1_{23}\,,
\end{align*}
which concludes the proof of the lemma.
\end{proof} 


\begin{proof}[Proof of Theorem~\ref{thm:FkSL3}]
We assume again that $k=3$. The general case follows with the same argument as in  Theorem~\ref{thm:F3SL2}.

As in the proof of Proposition~\ref{prop:FkSL3} we decompose $S=S_{0,4} = P'\cup P''$,
$\gamma_1 = P'\cap P''$. Also, we choose generators of $\pi_1(P')$, $\pi_1(P'')$, and
$\pi_1(P'\cap P'')$ as in the proof of Proposition~\ref{prop:FkSL3}. The peripheral elements of $S$ are
$\gamma_2$, $\gamma_3$, $\gamma_1\gamma_2$, and $\gamma_1\gamma_3$.

For a representation
$\rho\colon\pi_1(S)\to\SLL_3(\CC)$ 
 we let
$\rho'\colon\pi_1(P')\to\SLL_3(\CC)$ and $\rho''\colon\pi_1(P'')\to\SLL_3(\CC)$ denote the restriction of 
$\rho$ to $\pi_1(P')$ and $\pi_1(P'')$ respectively.

Let $[\rho]\in\R(S,\SLL_3(\CC))\setminus   \{ \uptau_{12\bar 1\bar{2}}= \uptau_{21\bar 2\bar 1} \}\cup
\{ \uptau_{13\bar 1\bar{3}}= \uptau_{31\bar 3\bar 1} \} \cup\{
  \varDelta^1_{23} = 0\}    $.
It follows from Lemma~\ref{Lemma:regtr} and Remark~\ref{Remark:good} that $\rho'$ and $\rho''$ are good, $\partial$-regular representations.
In what follows we let $\omega_{12}$ and $\omega_{13}$ denote the pullback of the symplectic form
$\omega_{P'}$ on $\R^*(P',\partial P',\SLL_3(\CC))_{\rho'}$ and
$\omega_{P''}$ on $\R^*(P'',\partial P'',\SLL_3(\CC))_{\rho''}$ respectively.

Given a basis $\mathbf{v}$ for $H^1(\gamma_1,\Ad\rho) $
we can choose lifts
$\mathbf{v}'\subset H^1(P',\Ad\rho) $, and $\mathbf{v}''\subset H^1(P'',\Ad\rho) $ which map to $\mathbf{v}$. By exactness there exists $\widetilde{\mathbf{v}}\subset H^1(S,\Ad\rho)$ which maps to $(\mathbf{v}',-\mathbf{v}'')$.

\begin{Lemma}
\label{lemma:iOmega}
Let   $\mathbf{u}$  a basis  for $H^0(\gamma_1,\Ad\rho) $ and $\mathbf{v}$  a basis  for $H^1(\gamma_1,\Ad\rho) $.
Then
\begin{equation}
 \label{mult:OmegaS} 
\Omega_S =\pm
\frac{  \langle \wedge\mathbf{u}, \wedge\mathbf{v} \rangle   }{(\nu_1 \wedge \nu_{23}) (\wedge\mathbf{v}\wedge\beta(\mathbf{u} )) }
\omega_{12}\wedge \omega_{13} \wedge\nu_2
 \wedge\nu_3\wedge\nu_{12}\wedge\nu_{13} \wedge \nu_1\wedge \nu_{23} .
 \end{equation}
\end{Lemma}
 
\begin{proof}[Proof of the lemma]
Chose $\mathbf{a}'$ a basis of $\ker ( H^1(P',\Ad\rho)\to  H^1(P'\cap P'',\Ad\rho))$ and 
$\mathbf{a}''$ a basis of $\ker ( H^1(P'',\Ad\rho)\to  H^1(P'\cap P'',\Ad\rho))$. 
Moreover, we can chose  lifts $\widetilde{(\mathbf{a}')}$, $\widetilde{(\mathbf{a}'')}\subset
H^1(S,\Ad\rho)$ which map under $j\colon H^1(S,\Ad\rho) \to H^1(P',\Ad\rho)\oplus H^1(P'',\Ad\rho)$ to 
$(\mathbf{a}',\mathbf{0})$ and $(\mathbf{0},\mathbf{a}'')$ respectively.

Then, by using  \eqref{eqn:lesp'p''}, $\mathbf{a}'\sqcup \mathbf v'$ is a basis for   $H^1(P',\Ad\rho)$, $\mathbf{a}''\sqcup \mathbf v''$ is a basis for   $H^1(P'',\Ad\rho)$
and $\widetilde{(\mathbf{a}')}\sqcup \widetilde{(\mathbf{a}'')}\sqcup \beta (\mathbf u)\sqcup \widetilde{\mathbf v}$ is a basis for $H^1(S,\Ad\rho)$.

The product formula applied to   \eqref{eqn:lesp'p''} yields:
\begin{multline*}
\Omega_S( \wedge\widetilde{(\mathbf{a}')}\wedge \widetilde{(\mathbf{a}'')}\wedge\beta( \mathbf{u})\wedge \widetilde{\mathbf{v}} )=\pm
\frac{\Omega_{P'}
( \wedge\mathbf{a}'  \wedge\mathbf{v}' )\, \Omega_{P''} (\wedge\mathbf{a}''  \wedge\mathbf{v}'' )  }{\operatorname{tor}(P'\cap P'',\Ad\rho,\mathbf u,\mathbf v)}
\\
=\pm
(\omega_{12}\wedge\nu_2\wedge\nu_{12}) (\wedge\mathbf{a}') 
(\omega_{13}\wedge\nu_3\wedge\nu_{13}) (\wedge\mathbf{a}'') 
\frac{\nu_1(\wedge\mathbf{v})^2}{ \operatorname{tor}(P'\cap P'',\Ad\rho,\mathbf u,\mathbf v) } .
\end{multline*}
The last equality follows since 
$d\,\tau_2$, $d\,\tau_{\bar{2}}$, $d\,\tau_{12}$, $d\,\tau_{\bar1\bar2}$, $d\,\tau_{1\bar2}$, $d\,\tau_{\bar1 2}$ 
vanish on each cocycle $v_i'$ of $\mathbf{v}'= (v_1',v_2')$, and 
$d\,\tau_3$, $d\,\tau_{\bar{3}}$, $d\,\tau_{13}$, $d\,\tau_{\bar1\bar3}$, $d\,\tau_{1\bar3}$, $d\,\tau_{\bar1 3}$
vanish on each cocycle $v_i''$ of $\mathbf{v}''= (v_1'',v_2'')$.

By Definition~\ref{Def:peripheralform},
$
{\nu_1(\mathbf{v})^2}/{ \operatorname{tor}(P'\cap P'',\Ad\rho,\mathbf u,\mathbf v) }
=
\pm\langle \wedge\mathbf{u}, \wedge\mathbf{v}  \rangle
$,
hence
\begin{multline*}
\Omega_S( \wedge\widetilde{(\mathbf{a}')}\wedge \widetilde{(\mathbf{a}'')}\wedge\beta( \mathbf{u})\wedge \widetilde{\mathbf{v}} )
=\pm
(\omega_{12}\wedge\nu_2\wedge\nu_{12}) (\wedge\mathbf{a}') 
(\omega_{13}\wedge\nu_3\wedge\nu_{13}) (\wedge\mathbf{a}'') \langle \wedge\mathbf{u}, \wedge\mathbf{v}  \rangle
 \\
=\pm
(\omega_{12}\wedge\nu_2\wedge\nu_{12}) (\wedge\mathbf{a}') 
(\omega_{13}\wedge\nu_3\wedge\nu_{13}) (\wedge\mathbf{a}'') \langle \wedge\mathbf{u}, \wedge\mathbf{v}  \rangle
\frac {(\nu_1 \wedge \nu_{23}) (\wedge\mathbf{v}\wedge\beta(\mathbf{u} )) }{(\nu_1 \wedge \nu_{23}) (\wedge\mathbf{v}\wedge\beta(\mathbf{u} )) }\\
=
\pm
\frac{ \langle \wedge\mathbf{u}, \wedge\mathbf{v}  \rangle}{  (\nu_1 \wedge \nu_{23}) (\wedge\mathbf{v}\wedge\beta(\mathbf{u} )) }
\omega_{12}\wedge \omega_{13} \wedge\nu_2
 \wedge\nu_3\wedge\nu_{12}\wedge\nu_{13} \wedge \nu_1\wedge \nu_{23} 
 \big(   \wedge\widetilde{(\mathbf{a}')}\wedge\widetilde{(\mathbf{a}'')} \wedge\beta(\mathbf{u})\wedge\widetilde{\mathbf{v}}   \big)\,.
\end{multline*}
The last equality follows since since $\beta({\mathbf{u}})$ is an infinitesimal bendings that vanish on $\nu_1$, and $\beta({\mathbf{u}})$ is in the kernel of $j$ (see \eqref{eqn:lesp'p''}).
Moreover, the bases $\widetilde{(\mathbf{a}')}$ and $\widetilde{(\mathbf{a}'')}$ map to
$(\mathbf{a}',\mathbf{0})$ and $(\mathbf{0},\mathbf{a}'')$ respectively.
\end{proof}

To conclude the proof of Theorem~\ref{thm:FkSL3}
%
we need to compute the quotient
$$
\frac{  \langle \wedge\mathbf{u}, \wedge\mathbf{v} \rangle   }{  
(\nu_1 \wedge \nu_{23}) (\wedge\mathbf{v}\wedge\beta(\mathbf{u} ))
}
$$
 As $\beta({\mathbf{u}})$ consist of infinitesimal bendings that vanish on $d\,\uptau_1$ and $d\,\uptau_{\bar 1}$,
 $$
 {  
(\nu_1 \wedge \nu_{23}) (\wedge\mathbf{v}\wedge\beta(\mathbf{u} ))
}={\nu_1(\wedge\mathbf{v})\nu_{23}(\wedge\beta(\mathbf{u} )) }.
 $$
Write 
$$
A=\rho(\gamma_1), \ 
B=\rho(\gamma_2), \textrm{ and } C=\rho(\gamma_3), 
$$
and 
$$
x=A-\frac{\operatorname{tr}(A) }{3}\operatorname{Id}\quad \textrm{ and }\quad  y=A^{-1}-\frac{\operatorname{tr}(A^{-1}) }{3}\operatorname{Id}.
$$
Hence $x,y  \in \mathfrak{sl}_3(\CC)
$
generate the  $A$-invariant subspace by Lemma~\ref{Lemma:reginv}.
By the natural identification $H^0(\gamma_1,\Ad\rho)\cong \mathfrak{sl}_3(\CC)^{\Ad_A}$, we chose $\mathbf u=\{x,y\}$.

To finish the proof of Theorem~\ref{thm:FkSL3}, we
assume semi-simplicity, so that  $H^1(\gamma_1,\Ad\rho  )\cong H^1(\gamma_1,\mathbb R)\otimes_\RR\mathfrak{sl}_3(\CC)^{\Ad_A}$ and we 
may chose $\mathbf{v}$ to be $\{x,y\}$ times the fundamental class.
Therefore
\begin{equation}
 \label{eqn:wedgeuv}
 \langle \wedge\mathbf{u}, \wedge\mathbf{v}  \rangle=\det
 \begin{pmatrix}
   \operatorname{tr}(x^2) & \operatorname{tr}(x y) \\
    \operatorname{tr}(x y) &    \operatorname{tr}(y^2)
 \end{pmatrix} .
\end{equation}
Next we compute $\nu(\wedge\mathbf{v})$. Write $\mathbf{v}=\{v_x,v_y\}$, where $v_x$ and $v_y$ are the infinitesimal deformations corresponding to
$x$ and $y$ respectively. 
Namely, the tangent vector to the infinitesimal paths 
\begin{equation}
 \label{eqn:infvxvyg1}
 \gamma_1\mapsto (\operatorname{Id}+\varepsilon x)A= A+\varepsilon x\,A \quad \text{ and }\quad  \gamma_1\mapsto (\operatorname{Id}+\varepsilon y)A= A+\varepsilon y\, A \quad \text{ in $\CC[\varepsilon]/\varepsilon^2$.}
\end{equation}
These infinitesimal deformations evaluated at $\gamma_1^{-1}$ are, respectively,
\begin{equation}
 \label{eqn:infvxvyg1bar}
\gamma_1^{-1}\mapsto A^{-1}(\operatorname{Id}-\varepsilon x)= A^{-1}-\varepsilon A^{-1}\, x \quad \text{ and }\quad 
\gamma_1^{-1}\mapsto A^{-1}(\operatorname{Id}-\varepsilon y)= A^{-1}-\varepsilon  A^{-1}\,y \quad 
\text{ in $\CC[\varepsilon]/\varepsilon^2$.}
\end{equation}
Thus, $d\, \uptau_1(v_x)=\operatorname{tr}(x\, A)$, and as $\operatorname{tr}(x)=0$, $\operatorname{tr}(x\, A)=\operatorname{tr}(x\, A- \frac{\uptau_1}{3}x)=\operatorname{tr}(x\, x)$.
By the very same argument, $\operatorname{tr}(y\, A)=\operatorname{tr}( A^{-1}x)=\operatorname{tr}(x\, y)$ and $\operatorname{tr}(A^{-1}\, y)=\operatorname{tr}(y^2)$, and 
\eqref{eqn:infvxvyg1} and \eqref{eqn:infvxvyg1bar} yield
\begin{equation}
\label{eqn:dt1dt2vxvy}
\begin{array}{ll}
  d\, \uptau_1(v_x)=\operatorname{tr}(x^2)\,,  \qquad\qquad & d\, \uptau_1(v_y)=\operatorname{tr}(x\, y)\,, \\
 d\, \uptau_{\bar 1}(v_x)=-\operatorname{tr}(x\, y)\,, &   d\, \uptau_{\bar 1}(v_y)=-\operatorname{tr}(y^2) .
\end{array}
\end{equation}
From  \eqref{eqn:wedgeuv} and \eqref{eqn:dt1dt2vxvy} we have
\begin{equation}
 \label{eqn:nu1}
d\, \uptau_1\wedge d\, \uptau_{\bar 1} (\wedge \mathbf v)= \pm \langle \wedge\mathbf{u}, \wedge\mathbf{v}  \rangle .
\end{equation}
In addition, by Lemma~\ref{lemma:nuDelta}
\begin{equation}
\label{eqn:nu23beta}
  \nu_{23} (  \beta(\mathbf{u} ))= \sqrt{-3}\, d\, \uptau_{23}\wedge d\, \uptau_{\bar 2\bar 3} (  \beta(\mathbf{u} )) = \pm\sqrt{-3} \varDelta_{23}.
\end{equation}
Hence,  as $\nu_1=\sqrt{-3 } \, d\, \uptau_{1}\wedge d\, \uptau_{\bar 1}$,  by \eqref{eqn:nu1} and \eqref{eqn:nu23beta}:
$$
\frac{  \langle \wedge\mathbf{u}, \wedge\mathbf{v} \rangle   }{  
(\nu_1 \wedge \nu_{23}) (\wedge\mathbf{v}\wedge\beta(\mathbf{u} ))
}
 =
\frac{  \langle \wedge\mathbf{u}, \wedge\mathbf{v} \rangle   }{\nu_1(\wedge\mathbf{v})\nu_{23}(\wedge\beta(\mathbf{u} )) }
=\pm\frac{1}{ 3\varDelta^1_{23}} .
$$
Now the volume formula follows from  Lemma~\ref{lemma:iOmega},  the last equation,  and
the expression of the symplectic forms $\omega_{12}$ and $\omega_{13}$ in \eqref{eqn:symplecticSL3}.
\end{proof}

\begin{footnotesize}
\bibliographystyle{plain}
\bibliography{vol}
%
%
%
%
%
%
%
%
%
%
\end{footnotesize}

\end{document}